\documentclass[12]{amsart}%
\usepackage{amssymb}
\usepackage{fullpage}
\newcommand{\PP}{{\mathfrak P}}
\newcommand{\XX}{{\mathfrak X}}
\newcommand{\Pp}{{\mathfrak P}}
\newcommand{\bb}{{\mathfrak b}}
\newcommand{\aaa}{{\mathfrak a}}
\newcommand{\Qq}{{\mathfrak Q}}
\newcommand{\UU}{{\mathfrak U}}
\newcommand{\Uu}{{\mathfrak U}}
\newcommand{\Aa}{{\mathfrak A}}
\newcommand{\Bb}{{\mathfrak B}}
\newcommand{\Cc}{{\mathfrak C}}
\newcommand{\Dd}{{\mathfrak D}}
\newcommand{\calA}{{\mathcal A}}

\newcommand{\calL}{{\mathcal L}}

\newcommand{\calV}{{\mathcal V}}

\newcommand{\Z}{{\mathbb Z}}

\newcommand{\F}{{\mathbb F}}
\newcommand{\N}{{\mathbb N}}
\newcommand{\ord}{\mbox{ord}}
\newcommand{\pp}{{\mathfrak p}}

\newcommand{\qq}{{\mathfrak q}}
\newcommand{\ttt}{{\mathfrak t}}
\newcommand{\nn}{{\mathfrak n}}
\newcommand{\dd}{{\mathfrak d}}

    \newcommand{\BB}{{\mathfrak B}}
    \newcommand{\DD}{{\mathfrak D}}

\newcommand{\rr}{{\mathfrak r}}

       \newcommand{\CC}{{\mathfrak C}}
         \newcommand{\YY}{{\mathfrak Y}}
\newcommand{\cc}{{\mathfrak c}}

\newcommand{\calE}{{\mathcal E}}
\font\te=eufm10
\newcommand{\mte}[1]{\mbox{\te {#1}}}

\newtheorem{theorem}{Theorem}[section]
\newtheorem{lemma}[theorem]{Lemma}
\newtheorem{corollary}[theorem]{Corollary}
\newtheorem{proposition}[theorem]{Proposition}

\theoremstyle{definition}
\newtheorem{definition}[theorem]{Definition}

\theoremstyle{remark}
\newtheorem{remark}[theorem]{Remark}

\newtheorem{notation}[theorem]{Notation}
\newtheorem{notationassumption}[theorem]{Notation and Assumptions}

\newcommand{\comment}[1]{}

\usepackage{color} 

    \input xy
\xyoption{all}
\begin{document}%
\bibliographystyle{alpha}%
\title[Hilbert's Tenth Problem over Function Fields of Positive
Characteristic]{Hilbert's Tenth Problem over Function Fields of
  Positive Characteristic Not Containing the Algebraic Closure of a
  Finite Field}%
\author{Kirsten Eisentr\"ager} \address{Department of Mathematics\\
  The Pennsylvania State University\\ University Park, PA 16802, USA.}
\author{Alexandra Shlapentokh} \address{Department of Mathematics\\
  East Carolina University\\ Greenville, NC 27858, USA.}
\begin{abstract} We prove that the existential theory of any function
  field $K$ of characteristic $p> 0$ is undecidable in the language of
  rings provided that the constant field does not contain the algebraic
  closure of a finite field.  
 We also extend the undecidability proof for function fields of higher 
  transcendence degree to characteristic 2 and show that the first-order theory of {\bf any} function
  field of positive characteristic is undecidable in the language of
  rings without parameters.\end{abstract}

\keywords{Undecidability,  Hilbert's Tenth Problem}

\thanks{K.\ Eisentr\"ager was partially supported by
  National Science Foundation grant DMS-1056703.  A.~Shlapentokh
  was partially supported by National Science Foundation grant DMS-1161456.}
\maketitle

\section{Introduction}\label{intro}
Hilbert's Tenth Problem in its original form was to find an algorithm
to decide, given a polynomial equation $f(x_1,\dots,x_n)=~0$ with
coefficients in the ring $\Z$ of integers, whether it has a solution
with $x_1,\dots,x_n \in \Z$. Matiyasevich \cite{Mat70}, building on
earlier work by Davis, Putnam, and Robinson \cite{DPR61}, proved
that no such algorithm exists, i.e.\ Hilbert's Tenth Problem is
undecidable.

Since then, analogues of this problem have been studied by asking the
same question for polynomial equations with coefficients and solutions
in other recursive commutative rings. A recursive ring is a
countable ring for which there is an algorithm to determine what the
elements of the ring are and such that the graphs of addition and multiplication
are also recursive.  Perhaps the most important unsolved question in
this area is Hilbert's Tenth Problem over the field of rational
numbers which, at the moment, seems out of reach.

The function field analogue of Hilbert's Tenth Problem in positive
characteristic turned out to be much more tractable. Hilbert's Tenth
Problem is known to be undecidable for the function field $K$ of a
curve over a finite field~\cite{Ph3,V,Sh13,Eis}. We also have
undecidability of Hilbert's Tenth Problem for certain function fields
over possibly infinite constant fields of positive
characteristic~\cite{Sh15, Sh30, Eis,K-R1}. The results of
\cite{Eis} and \cite{Sh15} also generalize to higher transcendence
degree (see \cite{Sh24} and \cite{Sh30}) and give undecidability of
Hilbert's Tenth Problem for finite and some infinite extensions of
$\F_q(t_1,\dots,t_n)$ with $n\geq 2$. In \cite{Eis2012} the
problem was shown to be undecidable for finite extensions of $k(t_1,
\dots, t_n)$ with $n \geq 2$ and $k$ algebraically closed of odd
characteristic.

So all known undecidability results for Hilbert's Tenth Problem in
positive characteristic either require that the constant field not be
algebraically closed or that we are dealing with a function field in
at least 2 variables. The big open question that remains is whether
Hilbert's Tenth Problem for a one-variable function field over an
algebraically closed field of constants is undecidable.  In this paper
we will shrink the window of the ``unknown'' almost precisely to the
question above by proving the following theorems (we separate the
countable and uncountable cases).
\begin{theorem}
\label{countable}
If $K$ is any countable function field not containing the algebraic
closure of a finite field, then Hilbert's Tenth Problem is not
solvable over $K$.
\end{theorem}
\begin{theorem}
\label{uncountable}
If $K$ is any function field of positive characteristic not containing
the algebraic closure of a finite field, then there exists a finitely
generated subfield $K_0 \subseteq K$ such that there is no algorithm
to determine whether a polynomial equation with coefficients in $K_0$
has solutions in $K$.
\end{theorem}
In \cite{ES}, the authors proved that the first-order theory of any
function field not equal to a function field of transcendence degree
at least 2 and characteristic 2 in the language of rings without
parameters is undecidable.  In this paper we prove the result in the
missing case to show that the following theorem is true.
\begin{theorem}
\label{first-order}
The first-order theory of any function field of positive characteristic in the language of rings without parameters is undecidable.
\end{theorem}

To explain the idea of the proof we need the notion of a Diophantine
(or existentially definable) set. Given a commutative integral domain
$R$ and a positive integer $k$, we say that a subset $A \subset R^k$
is {\it Diophantine} over $R$ or is {\it existentially definable} over
$R$ in the language of rings if there exists a polynomial
$f(t_1,\ldots,t_k,x_1,\ldots,x_n)$ with coefficients in $R$ such that
for any $k$-tuple $\bar a=(a_1,\ldots,a_k) \in R^k$ we have that $\bar
a \in A \iff \exists b_1,\ldots,b_n \in R:f(\bar a,
b_1,\ldots,b_n)=0.$ In this case, $f(t_1,\ldots,t_k,x_1,\ldots,x_n)$
is called a Diophantine definition of $A$ over $R$.
In general, if the fraction field of a recursive integral
domain is not algebraically closed, a system of polynomial equations
can always be effectively replaced by a single polynomial equation
without changing the relation~\cite[Chapter 1, \S
2, Lemma 1.2.3] {Sh34}.

The current methods for proving undecidability of Hilbert's Tenth
Problem for function fields $K$ of positive characteristic $p$ usually
require showing that the following sets are existentially definable in
the language of rings (or, equivalently, have a
Diophantine definition over $K$):
\[
P(K)=\{(x,x^{p^s}):x \in K, s \in \Z_{\geq 0}\},
\]
and for some nontrivial prime $\pp$ of $K$,
\[
\mbox{INT}(K,\pp)=\{x\in K: \ord_{\pp}x \geq~0\}.
\]

In \cite{ES} we showed that we can existentially define one of these
sets for a large class of fields: we proved that the set of $p$-th
powers $P(K)$ is existentially definable in {\em any} function field $K$ of
characteristic $p>2$ whose constant field has transcendence degree at
least one over $\F_p$. In \cite{PaPhVi} a uniform definition of $p$-th
powers was given for arbitrary function fields with the characteristic
``large enough'' compared to the genus of the field.  In this paper,
we show that the set $P(K)$ is existentially definable in {\it any} function
field $K$ of positive characteristic. In particular, we are finally able
to remove the assumption in \cite{Sh15} and \cite{Eis} that the
algebraic closure of $\F_p$ in $K$ should have an extension of degree
$p$. 

The most difficult part of our argument is defining $p^s$-th powers of
a special element $t$. In \cite{Sh15,Eis} we needed to assume that we
had suitable extensions of degree $p$ of the constant field to
conclude that a certain set of equations over $K$, which was satisfied
by an element $x \in K$, actually forced $x$ to be in the rational
function field $C_K(t)$ (Lemma 2.6 in \cite{Sh15} and Lemma 3.5 in
\cite{Eis}). Here $C_K$ denotes the constant field of $K$. This
argument does not work in our setting because our constant field can
be algebraically closed. Perhaps the most important new technical part
is contained in Lemma~\ref{le:boundedheight}, which is the key new
argument in Section~\ref{sec:new} that allows us to define $p^s$-th
powers of $t$ in arbitrary function fields of positive
characteristic.

The second set that is required to be existentially definable to prove the
undecidability of Hilbert's Tenth Problem is the set
$\mbox{INT}(K,\pp)$ defined above. In \cite{Sh15} the second author
showed that $\mbox{INT}(K,\pp)$ was existentially definable for some
non-trivial prime $\pp$ of $K$ over any function field whose constant
field was algebraic over a finite field and not algebraically closed,
and some higher transcendence degree constant fields not containing
the algebraic closure of a finite field.  In fact to show the
Diophantine undecidability of a function field $K$ of positive
characteristic, it is enough to give an existential definition of
$P(K)$ and of the set which we call $\mbox{INT}(K, \pp,t)$, where $t$
is a non-constant element of $K$ with $\ord_{\pp}t >0$, such that for
any $x \in K$ we have that $x \in \mbox{INT}(K, \pp,t) \Rightarrow
\ord_{\pp}x \geq 0$ and if $x \in k_0(t)$, where $k_0$ is the
algebraic closure of a finite field in $K$, and $\ord_{\pp} x \geq 0$,
then $x \in \mbox{INT}(K, \pp,t)$.

The structure of the remainder of the paper is as follows.  In
Section~\ref{proofidea} we explain how to derive the existential
undecidability of a function field $K$ of positive characteristic from
existential definitions of $P(K)$ and $\mbox{INT}(K, \pp,t)$ for some
non-trivial prime $\pp$ of $K$ and a non-constant element $t \in K$.
In Section~\ref{sec:rewrite} we discuss some general properties of Diophantine
definitions. In Section~\ref{sec:technical} we discuss some technical properties of
function fields of positive characteristic we will need to define
$P(K)$. In Section~\ref{sec:special} we give an existential definition of $P(K)$, and
in Section~\ref{sec:integralsubsets} we give an existential definition of $\mbox{INT}(K,
\pp,t)$.  Finally, in Section~\ref{sec:firstorder} we use the existential definition of
$P(K)$ to obtain the first-order results in Theorem~\ref{first-order}.

\section{From $p$-th powers and integrality at a prime to Diophantine undecidability}\label{proofidea}
We start with defining a relation on positive integers.
\begin{definition}
  For $m, n \in \Z_{>0}$ and $p$ a rational prime number, define $n \mid_p
  m$ to mean $m = np^s$ for some $s \in \Z_{\geq 0}$.
\end{definition}
In \cite{Ph5} Thanases Pheidas proved that the existential theory of
$(\Z_{>0}, +, \mid_p)$ is undecidable by showing that multiplication of
positive integers is definable using ``+'' and ``$\mid_p$''.
That means there is no uniform algorithm that, given a system of
equations over the positive natural numbers with addition and
$\mid_p$, determines whether this system has a solution or not. When
$P(K)$ and $\mbox{INT}(K,
\pp,t)$ are existentially definable, we can
reduce this problem to Hilbert's Tenth Problem over $K$ and prove that
Hilbert's Tenth Problem over $K$ must be undecidable.

To do this we define a map $f$ from the positive integers to subsets of $K$ by
associating to an integer $n$ the subset $f(n) = \{ x \in
\mbox{INT}(K, \pp,t) : \ord_{\pp}x =n\}.$ Then the equation
$n_3=n_1+n_2$ $ (n_i \in \Z_{>0})$ is equivalent to the existence of
elements $z_i \in f(n_i)$ with $z_3=z_1\cdot z_2$.

To ensure that we are only constructing equations over $K$ with $z_i$
elements of positive order (to obtain elements in $\Z_{>0}$ under the map
$K \to \Z$ that maps $z_i$ to $\ord_{\pp}(z_i)$) we add the condition that
$\ord_{\pp}(z_i/t) \in \mbox{INT}(K, \pp,t)$.

 We also have that for
positive integers $n,m$, 

\begin{align*} n \mid_p m &\iff \exists s \in \N m = p^s n\\
&\iff \exists x \in f(n)\;  \exists y \in f(m) \;\exists s \in \N\;
(\ord_{\pp} y = p^s \ord_{\pp} x).
\end{align*}
This equivalence can be seen by letting $x=t^n$ and $y = t^m$.

But the last formula is equivalent to
\[
\exists x \in f(n) \;\exists y \in f(m)\; \exists w \in K \;\exists s \in \N\;
w= x^{p^s} \text{ and } \{ w/y, y/w\} \subset   \mbox{INT}(K, \pp,t).
\]

Saying that both $w/y$ and $y/w$ are in $\mbox{INT}(K, \pp,t)$ simply
means that they have the same order at $\pp$.

\comment{

 The longer
(and more transparent) version of this assertion can be stated as the
following proposition.
\begin{proposition}
\label{Pheidas}
There is no algorithm which if given a finite collection of linear polynomials  
\[
L_i(x_1,\ldots,x_m), M_i(x_1,\ldots,x_m), N_i(x_1,\ldots,x_m), i=1,\ldots, r
\]
  in variables $x_1, \ldots,x_m$ and coefficients in $\Z$ can determine if the following system of equations 
\begin{equation}
\label{linsys}
\left \{
\begin{array}{c}
L_i(x_1,\ldots,x_m)=0\\
M_i(x_1,\ldots,x_m) \mid_p N_i(x_1,\ldots,x_m)
\end{array}
\right .
\end{equation}
has solutions in positive integers.  
\end{proposition}
Let a function field $K$ of positive characteristic $p$ and a
non-trivial prime $\pp$ of $K$ be given, and assume $\mbox{INT}(K,
\pp,t)$ and $P(K)$ have Diophantine definitions over $K$ (or are
existentially definable over $K$).  Now if we show that given an
arbitrary linear system \eqref{linsys}, one can algorithmically
produce a system of equations over $K$ (with coefficients in a fixed
finitely generated subfield) such that this system has solutions over
$K$ if and only if the system \eqref{linsys} has solutions in positive
integers, then we prove that Diophantine definability of
$\mbox{INT}(K, \pp,t)$ and $P(K)$ implies Diophantine undecidability
of $K$.  Our translation of linear equations over $\Z$ into polynomial
equations over $K$ will proceed as follows.  A linear equation
\begin{equation}
\label{lin}
w=L(x_1,\ldots,x_m)=\sum_{j=1}^m a_jx_j=0
\end{equation}
 will be replaced by equations corresponding to the statements below:
 \begin{equation}
 \label{prod}
 W=\prod_{j=1}^m Y_j^{a_j}, Y_j \not = 0
 \end{equation}
 and
 \begin{equation}
 \label{int1}
  \frac{Y_j}{t}, W, W^{-1} \in INT(K, \pp, t).
\end{equation}
  A $p$-divisibility condition 
  \begin{equation}
  \label{pdiv}
  x|_p y
  \end{equation} will be replaced by equations corresponding to 
  \begin{equation}
  \label{P(K)}
  (X, Y) \in P(K), 
  \end{equation}
  and
  \begin{equation}
  \label{int2}
  \frac{X}{t},\frac{Y}{t} \in INT(K, \pp, t).
  \end{equation}
  Now it is enough to make sure that \eqref{lin} has solutions in
  $\Z_{> 0}$ if and only if \eqref{prod} and \eqref{int1} have
  solutions in $K$, and that \eqref{pdiv} has solutions in $\Z_{> 0}$
  if and only if \eqref{P(K)} and \eqref{int2} have solutions in $K$.
  So suppose now that \eqref{lin} holds, i.e.\ we have $x_1, \ldots,
  x_n \in \Z_{> 0}$ with $w=\sum_{j=1}^m a_jx_j=0$ and $a_i\in \Z$.
  In this case, let $Y_j = t^{x_j}$ and let
  \[
  W =\prod_{j=1}^m Y_j^{x_j}=t^{\sum_{j=1}^ma_jx_j}.
  \]
  Since, $\ord_{\pp}t >0, x_j \geq 0,$ and $\sum_{j=1}^m a_jx_j = 0$,
  we have that \eqref{int1} holds for the specified values of $Y_j$
  and $W=1$.  Now assume that \eqref{prod} and \eqref{int1} hold.  Let
  $x_j = \ord_{\pp}Y_j >0$, and observe that $\ord_{\pp}W \geq 0,
  \ord_{\pp}W^{-1} \geq 0$ implies that $\ord_{\pp}W=0$.  Thus,
  $\sum_{j=1}^m a_jx_j = 0$.
  
  Next suppose \eqref{pdiv} holds.  In this case let $X=t^x$, and $Y =t^y$.  Since $y = xp^k$, we see that \eqref{P(K)} holds.  Further, since $\ord_{\pp}t >0$, and $x, y \in \Z_{> 0 }$, we also have that \eqref{int2} holds.  Finally suppose that $(X,Y) \in P(K)$ and \eqref{int2} holds.  In this case, let $0 < \ord_{\pp}X =x$ and let $0 < \ord_{\pp}Y =y$.  Since $Y = X^{p^k}$ we also have that $ x|_p y$.  \\
  }
  
  We have now proved the following proposition.
  \begin{proposition}
  \label{noalg}
  If $K$ is a function field of positive characteristic $p$ over a
  field of constants $k$, $\pp$ is a non-trivial valuation (or prime)
  of $K$, $t \in K\setminus k$, has a positive order at $\pp$, and
  $P(K)$ and $\mbox{INT}(K, \pp, t)$ are existentially definable over
  $K$, then for some finitely generated subfield $K_0$ of $K$, there
  is no algorithm to determine whether an arbitrary polynomial
  equation in several variables and with coefficients in $K_0$ has
  solutions in $K$.
  \end{proposition}
  
\section{Rewriting equations over finite extensions}
\label{sec:rewrite}
In constructing Diophantine definitions it is often convenient to work
over finite extension of the given field, sometimes in fixed
extensions and
sometimes in extensions of bounded degree.  The theorem below and its corollaries
allow us to do this.  The next theorem is Lemma B.7.5 in the Number
Theory Appendix of \cite{Sh34}.

\begin{theorem}
\label{bounddegree}
Let $K$ be a field, let $\tilde K$ be the algebraic closure of $K$, let 
\[
g(X,Z_1,\ldots,Z_{n_1}),
\]
\[
 f(T_1,\ldots, T_n,X_1,\ldots,X_{n_2}, Y_1,\ldots, Y_{n_3})
 \]
 be polynomials with coefficients in $K$, and let $A \subset K^n$ be
 defined in the following manner: $(t_1,\ldots, t_n) \in A$ if and
 only if there exist $z_1,\ldots,z_{n_1}, x_1,\ldots,x_{n_2} \in K, x
 \in \tilde K, y_1,\ldots,y_{n_3} \in K(x)$ such that
\[
g(x,z_1,\ldots,z_{n_1})=0 \land f(t_1,\ldots,t_n,x_1,\ldots,x_{n_2}, y_1,\ldots, y_{n_3})=0.
\]
In this case $A$ has a Diophantine definition over $K$.  Further,
there is a Diophantine definition of $A$ with coefficients depending
only on coefficients and degrees of $g$ and $f$ and it can be constructed
effectively from those coefficients.
\end{theorem}
The most often used versions of the theorem above are the following corollaries (though we will need the theorem also).
\begin{corollary}
\label{fixed}
Let $K$ be a field, let $G$ be a finite extension of $K$, let
$f(T,X_1,\ldots,X_{n_2}, Y_1,\ldots, Y_{n_3})$ be a polynomial with
coefficients in $K$, and let $A \subset K$ be defined in the following
manner: $t \in A$ if and only if there exist $x_1,\ldots,x_{n_2} \in
K, y_1,\ldots,y_{n_3} \in G$ such that
\[
f(t,x_1,\ldots,x_{n_2}, y_1,\ldots, y_{n_3})=0.
\]
In this case $A$ has a Diophantine definition over $K$.
\end{corollary}
\begin{corollary}
\label{probabove}
Let $G/K$ be a finite extension of fields and assume Hilbert's Tenth
Problem is unsolvable over $G$ (if $G$ is uncountable, then assume we
are considering equations with coefficients in a finitely generated
subfield of $G$). In this case Hilbert's Tenth Problem is unsolvable
over $K$ (as above, if $K$ is uncountable, then assume we are
considering equations with coefficients in a finitely generated
subfield of $K$).
\end{corollary}

\section{Technical preliminaries}
\label{sec:technical}
\setcounter{equation}{0}

\begin{notationassumption}\label{Natasha}
  In this section we go over or prove several technical facts we need
  to construct our existential definition of $p$-th powers.  We will
  initially work under the assumption that the constant field is
  algebraically closed.  This assumption will be removed later.  Below
  we use the following notation and assumptions.
\begin{enumerate}
\item By a {\em function field (in $1$ variable)} over a field $k$ we mean a field
$K$ containing $k$ and an element $x$, transcendental over $k$,
 such that $K/k(x)$ is a finite
algebraic extension. The algebraic closure of $k$ in $K$ is called the
constant field of $K$, and it is a finite extension of $k$.
\item Let $M$ be a function field of genus $g >0$ over an algebraically closed field of constants $F$ of characteristic $p>0$.

\item Let $t \in M$ be such that it is not a $p$-th power.  (Since the
  constant field is perfect, this assumption implies $M/F(t)$ is
  separable.)
\item A \emph{prime} of $M$ is an $F$-discrete valuation of $M$. 
\item The degree of a prime is the degree of its residue field over
  the field of constants.  Under our assumption that the constant
  field is algebraically closed, the degree is always one.
\item A divisor is an element of the free abelian group on
  the set of primes of $M$. We will denote the group law
  multiplicatively.
\item If $\mathfrak I$ is an integral (or effective) divisor, we will
  denote by $\deg \mathfrak I$ the degree of $\mathfrak I$, i.e.\ the
  number of primes in the product (counting multiplicity).
\item If $\mathfrak I$ is an integral divisor and $\pp$ is a prime,
  then $\ord_{\pp}\mathfrak I$ is the multiplicity of $\pp$ in the
  product.
\item If $\mathfrak I_1$ and $\mathfrak I_2$ are integral divisors, we
  write $\mathfrak I_1 \mid \mathfrak I_2$ ($\mathfrak I_1$ divides
  $\mathfrak I_2$) to mean that for all primes $\pp$ of $K$ we have
  that $\ord_{\pp}\mathfrak I_1 \leq \ord_{\pp}\mathfrak I_2$.
  Similarly for any prime $\pp$ of $M$ we write that $\pp \mid \mathfrak
  I_1$ ($\pp$ divides $\mathfrak I_1$) to mean $\ord_{\pp}\mathfrak
  I_1 >0$.
\item For $x \in M$, let $\nn(x)$ denote the zero divisor of $x$ and $\dd(x)$ the
pole divisor of $x$. Let $\DD(x)=\frac{\nn(x)}{\dd(x)}$ be the divisor of $x$. Let $H(x)$
denote the height of $x$, i.e.\ $\deg \dd(x) = \deg \nn(x)$, and if $\pp$ is a prime, let
$$
\ord_{\pp}(x)=\ord_{\pp}\mathfrak D(x)=\ord_{\pp}\mathfrak n(x)-\ord_{\pp}\mathfrak d(x).
$$

\item Since the extension $M$ over $F(t)$ is separable, we can
  define a global derivation with respect to $t$. Over $F(t)$, we use
  the usual definition of the derivative, and we use implicit
  differentiation to extend a derivation to the extension
  (see~\cite[p.\ 9 and p.\ 94]{Mason}). Given an element $x$ of $M$, the
  derivative with respect to $t$ will be denoted in the usual fashion
  as $x'$ or $\frac{dx}{dt}$. Observe that usual differentiation rules
  apply to the global derivation with respect to $t$.
\item For any prime $\pp$ of $M$, we can also define a local
  derivation with respect to the prime $\pp$. More specifically, if
  $\pi$ is any local uniformizing parameter with respect to $\pp$ (any
  element of $M$ which has order one at $\pp$), in the $\pp$-adic
  completion of $M$, every element $x$ of the field can be written as
  an infinite power series
$$
\sum_{i=m}^{\infty}a_i\pi^i
$$
with $m\in\Z$ and $a_i\in F$. Given this representation, we denote
$$
\frac{\partial x}{\partial \pp}=\sum_{i=m}^{\infty}ia_i\pi^{i-1}
$$
(see~\cite[p.\ 9 and p.\ 96]{Mason}). Observe that
$\ord_{\pp}(\frac{\partial x}{\partial \pp})$ is independent of the
choice of the local uni\-for\-mizing parameter.
\item For all primes $\pp$ of $M$, let
$$
d_t(\pp)=\ord_\pp\left(\frac{\partial t}{\partial \pp}\right).
$$

\item\label{VectorSpace} If $\mathfrak U=\frac{\mathfrak A}{\mathfrak
    B}$, where $\mathfrak A$ and $\mathfrak B$ are integral divisors,
  then, we will write
$$
L(\mathfrak U)=\{f\in M \mid \ord_{\pp}f\geq\ord_{\pp}\mathfrak{A}-\ord_\pp\mathfrak{B}\mbox{ for all primes }\pp \mbox{ of } M\}\cup\{0\}, 
$$
where $L(\mathfrak U)$ is a vector space over $F$, and $\ell(\mathfrak A)$ for the dimension of $L(\mathfrak A)$ over $F$.
\end{enumerate}
\end{notationassumption}

The following lemma gathers some general formulae we need in this section.

\begin{lemma}\label{Formulae}
$\left. \right.$
\begin{enumerate}
\item \label{FormulaOrd} Let $E$ be a finite degree subfield of a
  function field $K$. Let $\PP$ be a prime of $E$ and let
  $\pp_1,\dots,\pp_n$ be the primes in $K$ above $\PP$. Let
  $e(\pp_i/\PP)$ be the ramification index of $\pp_i$ over $\PP$. Let
  $f(\pp_i/\PP)$ be the relative degree of $\pp_i$ over $\PP$ (the
  degree of the extension of the residue field). We have
$$
[K:E]=\sum_{i=1}^ne(\pp_i/\PP)f(\pp_i/\PP).
$$
If the field of constants of $E$ is algebraically closed, the relative degrees will always be equal to one.
\item\label{RiemannRoch} (Riemann-Roch) Let $\mathfrak{U}
  =\frac{\mathfrak A}{\mathfrak B}$ be a ratio of integral divisors of
  $K$ such that $\deg{\Bb} -\deg{\Aa}=d \in \Z$.
\begin{enumerate}
\item\label{RR1} If $g=0$ and $d\geq0$ then $\ell(\mathfrak A)= d+1$;
\item\label{RR2} If $g>0$ and $0<d<2g-2$ then $\ell(\mathfrak A)\geq d-g+1$;
\item\label{RR3} If $g>0$ and $d=2g-2$ then $\ell(\mathfrak A)\geq g-1$;
\item\label{RR4} If $g>0$ and $d>2g-2$ then $\ell(\mathfrak A)= d-g+1$;
\end{enumerate}

\end{enumerate}
\end{lemma}
\begin{proof}
  For \eqref{FormulaOrd} see~\cite[Proposition 2.3.2, Theorem
  3.6.1]{Friednew}. For \eqref{RiemannRoch} see~\cite[Theorem
  5.6.2]{Koch}.
\end{proof}
The next lemma is an elementary result from linear algebra that we need to make use of the Riemann-Roch Theorem.
\begin{lemma}
\label{vs}
If $V$ is a vector space of dimension $n >0$ over an infinite field of
scalars and $\{V_1, \ldots, V_m\}$ is a finite collection of subspaces
of $V$, each of dimension $n-1$, then $V \setminus \left(\bigcup_{i=1}^m
V_i\right)$ contains infinitely many elements.
\end{lemma}
\begin{proof}
  Let $A$ be a countable subset of the field of scalars, let
  $\{a_1,a_2,\ldots\}$ be an ordering of $A$, and let $A_j =
  \{a_1,\ldots, a_j\}$. It is enough to show that
  \[
 C=\left\{\sum_{i=1}^nb_iv_i, b_i \in A \right\} \setminus (V_1 \cup \ldots \cup V_m)
 \]
  is infinite, and note further that it is enough to show that $C$ contains finite subsets of arbitrary size. Let
  \[
  C_r= \left\{\sum_{i=1}^nb_iv_i: b_i \in A_r\right\} \setminus (V_1 \cup \ldots \cup V_m)
\]
and observe that $C_r \subset C$ for all $r \in \Z_{>0}$.  Select a
basis $v_1,\ldots, v_n$ for $V$.  If we fix $i=1,\ldots,m$, then
without loss of generality, after possibly renumbering the elements
$v_1,\ldots,v_n$ (with the renumbering  depending on $i$), we can assume that for $j=1,\ldots, n-1$, each $V_i$ has a basis $\{w_{i,1},\ldots, w_{i,n-1}\}$, where
 $w_{i,j}=v_j +c_{i,j}v_n$.  Let
 \[
 u \in \left\{\sum_{i=1}^nb_iv_i: b_i \in A_r\right\}  \cap V_i.
 \] 
  In this case $u = \sum_{j=1}^{n-1}b_j w_{i,j}$, where $b_j \in A_r$.  Thus, 
  \[
  \left |\left \{\sum_{i=1}^nb_iv_i : b_i \in A_r \right \}  \cap V_i \right |  = r^{n-1}.
  \]
    Therefore 
    \[
    |C_r| \geq (r^n -mr^{n-1}) \rightarrow \infty \mbox{ as }r \rightarrow \infty.
    \]  
\end{proof}
Below is the first application of the Riemann-Roch Theorem we need.
\begin{lemma}
\label{Using RR}
If $\ttt$ is a prime of $M$ and $\Aa$ and $\Bb$ are two integral
relatively prime divisors of $M$, both also relatively prime to
$\ttt$, then there exists $y \in M$ such that $\displaystyle \dd(y)
=\ttt^{2g +1+ \deg \Aa}$ and $\nn(y)=\Aa\CC$, where $\CC$ is an
integral divisor relatively prime to $\Aa$ and $\Bb$.  Further, $\deg
\CC =2g+1$.
\end{lemma}
\begin{proof}
  Let $\displaystyle \UU= \frac{\Aa}{\ttt^{2g +1+ \deg \Aa}}$ and note
  that by Lemma \ref{Formulae}(\ref{RiemannRoch}) we have
  that $$\ell(\UU) = g+2 >0.$$ Further, let
  $\displaystyle\UU_1=\frac{\Aa}{\ttt^{2g + \deg \Aa}}$ and observe
  that $\ell(\UU_1) =g+1$ while $\calL(\UU_1) \subset \calL(\UU)$.
  Finally, let $\calA$ be the set of all primes $\rr$ of $M$ such that
  either $\ord_{\rr}\Aa\not = 0$ or $\ord_{\rr}\Bb\not = 0$ and let
  $|\calA|=m$.  Set
  $\displaystyle\UU_{i+1}=\frac{\Aa\rr_{i}}{\ttt^{\deg \Aa+ 2g+1}}$,
  where $\rr_i$ is the $i$-th element of $\calA$ under some
  enumeration of $\calA$.  Observe that by Lemma
  \ref{Formulae}(\ref{RiemannRoch}) again $\ell(\UU_{i+1})=g+1$ for
  $i=1,\ldots, m$ while $\calL(\UU_{i+1}) \subset \calL(\UU)$.  (We
  remind the reader that since the constant field of $M$ is
  algebraically closed all the primes are of degree 1.)  Now consider
\[
y \in L(\UU) \setminus \left(\bigcup_{j=1}^{m+1}L(\UU_j)\right). 
\]

Such a $y$ exists by Lemma \ref{vs}.  By construction, we have that
$\dd(y)=\ttt^{2g +1+ \deg \Aa}$ and $\nn(y)=\Aa\Cc$, where $\Cc$ is
relatively prime to $\Aa\Bb\ttt$.  Finally, $$\deg \Cc = \deg \dd(y) -
\deg(\Aa) = 2g+1 +\deg \Aa -\deg\Aa=2g+1.$$
\end{proof} 
We now specialize the lemma above to a particular divisor.
\begin{corollary}
\label{cor:usingRR}
Suppose $w \in M$ is an element whose divisor is of the form
$\displaystyle\frac{\XX\Aa^{p^a}}{\YY\BB^{p^a}}$, where $\XX,\YY, \Aa,
\BB$ are pairwise relatively prime integral divisors, and $\deg \YY \geq
\deg \XX$. Assume further that the prime $\ttt$ is a factor of $\YY$,
and that for some positive
constant $C$ we have $\deg(\YY) < C$. Then $\displaystyle w = \xi
\frac{z_1^{p^a}}{z_2^{p^a}}$, where $\ttt$ is the only pole of $z_1$
and $z_2$, $\xi$ does not have a zero or a pole at any prime which is
a zero of $z_1$ or $z_2$, and $H(\xi) < (p^a+1)(C+g+1)$.
\end{corollary}
\begin{proof}
  First of all observe that $\deg \YY -\deg \XX =p^a \deg \Aa - p^a
  \deg \Bb < C$.  Further, by Lemma \ref{Using RR}, there exist $z_1,
  z_2 \in M$ with divisors $\displaystyle \frac{\Aa\CC}{\ttt^{\deg
      \Aa+2g+1}}$, $\displaystyle \frac{\Bb\DD}{\ttt^{\deg \Bb+2g+1}}$,
  respectively, such that $\Aa, \Bb, \CC, \DD, \XX, \YY$ are all
  pairwise relatively prime and such that $$\deg \CC=\deg \DD= 2g+1.$$ Let
  $\displaystyle\xi=w\frac{z_2^{p^a}}{z_1^{p^a}}$.  In this case
\[
(\xi)=\frac{\XX\Aa^{p^a}}{\YY\BB^{p^a}}\frac{\Bb^{p^a}\DD^{p^a}}{\ttt^{p^a(\deg
    \Bb+2g+1)}}\frac{\ttt^{p^a(\deg
    \Aa+2g+1)}}{\Aa^{p^a}\CC^{p^a}}=\frac{\XX}{\YY}\frac{\DD^{p^a}\ttt^{p^a(\deg
    \Aa-\deg\Bb)}}{\CC^{p^a}},
\]
and therefore $$H(\xi) \leq \deg \XX +p^a \deg \DD + p^a \deg \Aa - p^a \deg \Bb \leq C +p^a(g+1)+ p^aC <(p^a+1)(C+g+1).$$
\end{proof}
The next two lemmas deal with the relationship between the derivatives (global and local) and order at a prime.
\begin{lemma}\label{Mason}
Let $x\in M$ and $\ttt$ be a prime of $M$. We have
\begin{enumerate}
\item\label{Mason1} $\ord_\ttt(\frac{\partial x}{\partial \pp})\geq\ord_\ttt(x)-1$; and
\item\label{Mason2} if $\ord_\ttt(x)\geq0$, then $\ord_\ttt(\frac{\partial x}{\partial \ttt})\geq0$.
\end{enumerate}
\end{lemma}
\begin{proof}
See~\cite[p.\ 9]{Mason}.
\end{proof}




\begin{lemma}\label{Ania}
Let $x\in M$ and let $\pp$ be a prime of $M$.
\begin{enumerate}
\item\label{gaga1} If $\ord_\pp(x)\geq0$, then
$
\ord_\pp(x')\ge\max(0,\ord_\pp(x)-1) -d_t(\pp).
$

\item\label{gaga2} If $\ord_\pp(x)<0$, then
$
\ord_\pp(x')\ge\ord_\pp(x)-1-d_t(\pp).
$
\end{enumerate}
\end{lemma}
\begin{proof}
By~\cite[p.\ 96]{Mason} we have that for any prime $\pp$ 
\begin{equation}\label{Michal}
\frac{\partial x}{\partial \pp}=\frac{dx}{dt}\frac{\partial t}{\partial \pp}.
\end{equation}
Hence if $\ord_\pp(x)\geq0$, then
$$
\ord_\pp(x')=\ord_\pp\left(\frac{dx}{dt}\right)=\ord_\pp\left(\frac{\partial x}{\partial \pp}\right)-\ord_\pp\left(\frac{\partial t}{\partial \pp}\right)\geq\max(0,\ord_\pp(x)-1)-d_t(\pp).
$$
If $\ord_\pp(x)<0$, then
$$
\ord_\pp(x')=\ord_\pp\left(\frac{dx}{dt}\right)=\ord_\pp\left(\frac{\partial x}{\partial \pp}\right)-\ord_\pp\left(\frac{\partial t}{\partial \pp}\right)\geq\ord_\pp(x)-1-d_t(\pp)
$$
by Lemma \ref{Mason}.
\end{proof}

\begin{lemma}
\label{ramification}
For any element $z \in M\setminus M^p$ there are at most $2g-2 +
2H(z)$ primes $\ttt$ of $M$ such that $d_z(\ttt)>0$, and for all
$M$-primes $\ttt$ we have that $d_z(\ttt) \leq 2g-2 +2H(z)$.
\end{lemma}
\begin{proof}
By~\cite[Equation (5) p.\ 10]{Mason}, we have
$$
\sum_{\ttt} d_z(\ttt)=\sum_{\ttt}\ord_\ttt\left(\frac{\partial
    z}{\partial\ttt}\right)=2g-2,
$$
since $z$ has non-zero global derivative.  By Lemma~\ref{Mason}, if
$\ord_{\ttt}\partial z/\partial \ttt <0$, then $\ord_{\ttt}z <0$.
Thus,
\[
\sum_{\ord_{\ttt}\partial z/\partial \ttt <0}|\ord_{\ttt}\partial z/\partial \ttt| \leq  \sum_{\ord_{\ttt} z< 0}(|\ord_{\ttt}z| +1) \leq 2H(z).
\]
Further,
\[
\sum_{\ord_{\ttt}\partial z/\partial \ttt >0}\ord_{\ttt}\partial z/\partial \ttt =  2g-2 + \sum_{\ord_{\ttt}\partial z/\partial \ttt <0}|\ord_{\ttt}\partial z/\partial \ttt|  \leq 2g-2+ 2H(z).
\]
\end{proof}
The last technical lemma of this section deals with the case of $d_t(\ttt)=0$.

\begin{lemma}
\label{noram}
If $\ttt$ is a prime of $M$ which is unramified over $F(t)$, and  $\ttt$ is not a pole of $t$, then $d_t(\ttt) =0$.
\end{lemma}

\begin{proof}
  If $\ttt$ is unramified over $F(t)$ and it is not a pole of $t$,
  then for some $c \in F$ we have that $\ord_{\ttt}(t-c) =1$.  Thus if
  we set $t-c=\pi$, a $\ttt$-adic expansion of $t$ is of the form
  $c+\pi$, and the derivative of that expression with respect to $\pi$
  is 1, implying $\ord_{\ttt}\partial t/\partial \ttt=0$.
\end{proof}

From this lemma we derive a corollary which will help us construct
$p$-th powers.  The proof of the corollary follows directly from Lemma
\ref{Ania} and Lemma \ref{noram}

\begin{corollary}
\label{cornoram}
If $\ttt$ is a prime of $M$ which is unramified over $F(t)$, and $\ttt$ is not
a pole of $t$, then for any $x$ which is integral at $\ttt$ we have
$\ord_{\ttt} dx/dt \geq \max(0,\ord_{\ttt}x -1)$.

\end{corollary}


\section{Defining $p$-th powers}
\label{sec:special}
\setcounter{equation}{0}

In this section we construct an existential definition of the set
$P(K)$ of $p^s$-th
powers.  We start under the assumption that the field of constants is
algebraically closed and remove this assumption later in Subsection
\ref{arbit}.  As for any construction of a Diophantine definition, the
construction of the set of $p^s$-th powers
has two main parts: one part consists of showing that the given
equations have at most $p^s$-th powers as their solutions. For the
second part we have to show that elements of $P(K)$ are in fact
solutions.  As it turns out, the second part is trivial in our case and
we will delay it until the very end in Lemma \ref{satisfy}.  The bulk
of the section below will be devoted to showing that the only elements
that can be solutions of our equations are the elements of $P(K)$. We do this in several
steps. As in earlier papers, the first part will be devoted to dealing
with the $p^s$-th powers of a particular element, the second part will
deal with $p^s$-th powers of elements of the field with simple zeros
and poles, and finally the third part will address the case of
arbitrary elements.

\subsection{Defining $p$-th powers of a particular element}
\label{sec:new}

The most difficult part of the argument is the first one: defining
$p^s$-th powers of a particular element.  We outline this construction
before proceeding with the technical details. We first fix a
non-constant element $t$ of $M$ satisfying certain conditions
described below and let $a=1$ or $a=2$ depending on the characteristic
of the field.  Next we let $z \in M$ be such that the equations in
Lemma~\ref{le:boundedheight} below are satisfied with $w=z+c$ for a
sufficiently large number of $c$'s. Here the requisite number
ultimately depends on the genus $g$ of $M$ only.  The equations lead
us to conclude that either $z$ has ``bounded'' height (with the bound
ultimately depending on $g$ only) or that the divisor of $z+c$ is a
$p^a$-th power of another divisor for all $c$'s.  In the first case we
use the equations from Proposition~\ref{prop:weak} and
Corollary~\ref{finalpart1} to conclude that $z \in F(t)$, and
Propositions~\ref{prop:version1} and \ref{prop:below} to conclude that
either $z$ is a $p^a$-th power of another field element or it is equal
to $t$.  In the second case we use Lemmas \ref{le:coefficients},
\ref{lemma:norms} and \ref{lemma:divpth} to conclude that $z$ is a
$p^a$-th power of another element.
  
Once we conclude that $z$ is a $p^a$-th power of another element, we
``take a $p^a$-th root'' of all the equations (or perform a form of
descent) and re-examine the resulting equations.  Since this descent
cannot continue forever, at some point we can conclude that $z$ was a
$p^s$-th power of $t$ for some $s \in \Z_{>0}$.

\begin{notationassumption}%
\label{notation: pthpowers}
Below we add to the notation and assumptions above.
\begin{itemize}%
\item  Let $a =1$, if $p >2$ and let $a=2$, if $p=2$.
\item Let $t \in M$ be such that no zero or pole of $t$ is ramified in
  the extension $M/F(t)$ or alternatively all zeros and poles of $t$
  are simple.
\item Denote the zero divisor of $t$ by $\Pp$ and the pole divisor by
  $\Qq$.  (We will also use the same notation for the primes which are
  the zero and the pole of $t$ in $F(t)$.)  Let $\Pp=\prod_i \pp_i,
  \Qq=\prod_i \qq_i$ be the factorization of $\Pp$ and $\Qq$ into
  distinct prime divisors of $M$.
\item Let $\calE$ be the set of all primes ramifying  in the extension $M/F(t)$ and let $e$ be the size of the set.
\item Let $M^G$ be the Galois closure of $M$ over $F(t)$.  Let $i_G=[M^G:M]$.
\item For $j=1,\ldots, k$, let $\sigma_j: M \longrightarrow M^G$ be an embedding of $M$ into $M^G$ over $F(t)$.
\item Let $\Omega=\{ \omega_1=1, \ldots, \omega_k \}$ be a basis of $M$ over
  $F(t)$. 
  \item Let $H_{\Omega}=\max\{H_{M^G}(\omega_i), i=1,\ldots,k\}$, where $H_{M^G}$ is the height in $M^G$.
\item Let $C=H(t)$.  (In Lemma \ref{heigtoft} we show that we can always assume that $C \leq \max(1, 2g-1)$.)
\item Let $C_1=2g-2 + (p^a+1)(C+g+1)$.
\item Let $\displaystyle C_2=\frac{ 2g-1 + (p^a +1)(C+g+1) }{p^a-1}$.
\item Let $C_3= C +p^aC_1C_2$.
\item Let $C_4=k!k^kH_{\Omega}C_3$
\item Let $C_5=C_4 + 2e +2k + 4H(t) +2$.
\item Let $C(F)=\{c_0, \dots,c_{C_5}\} \subset F_0$ -- the algebraic
  closure of $\F_p$ in $F$, let $d_{i,j}=c_i^{p^j}$, let
\[
V_i =  \{c_i^{p^k}, k\in \Z_{\geq 0}\},
\]
 and let $r_i :=|V_i|$.  Assume for
  $i \not=j$, for any $n_i, n_j \in \Z_{\not = 0}$ we have that
  $c_i^{n_i}\not = c_j^{n_j}$.  (The existence of $C(F)$ in a finite
  extension of $\F_p$ can be seen from the following inductive
  argument: let $c_0 \in \F_p, c_i \in F_0 \setminus
  \F_p(c_0,\ldots,c_{i-1}$).)
\item If $z \in M \setminus F$, let $C_z \subset C(F)$ be the such
  that $c \in C_z$ if and only if $c \in C(F)$ and $z-c^{p^s}$ does
  not have a zero at any prime which is a zero or a pole of $t$ or at
  any prime ramified in $M/F(t)$ for any positive integer $s$.
\item For a $w \in M$, let $\calV(w)$ be a set of primes of $F(t)$ satisfying the following requirements.
\begin{enumerate}
\item \label{it:1} Each prime of $\calV(w)$ is unramified in the extension $M/F(t)$.
\item \label{it:2} $w$ is integral at all primes of $\calV(w)$.
\item \label{it:3} $\{ \omega_1, \ldots, \omega_n \}$ is a local
  integral basis with respect to every prime of $\calV(w)$ or in other
  words every $\mathfrak A \in \calV(w)$ is relatively prime to the the
  discriminant of the basis.
\item The size of $\calV(w)$ is greater than $C_4$.
\end{enumerate}
\end{itemize}%
\end{notationassumption}
We start with a sequence of preliminary lemmas, some of them coming from earlier papers and included here for the convenience of the reader.

 \begin{lemma}[Essentially \cite{Sh34}, Lemma 8.2.10]
\label{cor:C(w)}%
For any $u, w \in M \setminus F$, we have that $|C_w|$ and $|C_w \cap C_u|$
contain more than $C_4 +2k+2$ elements.
\end{lemma}
Compared to the lemma in the citation we need more constants, so we
start with more constants, and in our case $C_4+2k$ replaces
$n$, but otherwise the argument is the same.
\comment{
\begin{proof}%
Consider the following table.
\[%
\left [
\begin{array}{ccccc}%
w-c_1&w-c_1^p& \ldots&w-c_1^{p^j}&\ldots\\%
\vdots&\vdots&\vdots&\vdots&\vdots\\%
w-c_l&w-c_l^p& \ldots&w-c_l^{p^j}&\ldots\\%
\end{array}%
\right ]
\]%

Observe that by assumption on the elements of $C(F)$ no two rows share
an element, and the difference between any two elements of the table
is constant and therefore any prime that occurs in the zero divisor of
an element in one row, does not occur in the zero divisor of any
element in any other row. Thus, elements of at least $C_5-e - 2H(t)$
rows have no zero at any element of ${\mathcal E}$, at any $\pp_i$ or
$\qq_i$, and consequently, $$|C_w| \geq C_5-e - 2H(t) .$$ Next
consider a table
\[%
\left [
\begin{array}{ccccc}%
u-b_1&u-b_1^p& \ldots&u-b_1^{p^j}&\ldots\\%
\vdots&\vdots&\vdots&\vdots&\vdots\\%
u-b_{|C_w|}&u-b_{|C_w|}^p& \ldots&u-b_{|C_w|}^{p^j}&\ldots\\%
\end{array}%
\right ]
\]%
where $b_i \in C_w$.  By an analogous argument, at least
$|C_w|-e - 2H(t)$  rows of this table contain no element with a zero at
any valuation of ${\mathcal E}$,  any $\pp_i$ or $\qq_i$.  Thus, at least $C_5-2e-4H(t) >C_4+2k$ elements are contained in
$C_w \cap C_u$.\\
\end{proof}
}

The next lemma is an elementary fact concerning valuations, and we state it without proof.
\begin{lemma}
\label{zerosandpoles}
For any non-constant element $z \in M$ and any constants $c\rq{} \not
=c$ we have that the zeros of $\displaystyle \frac{z-c\rq{}}{z-c}$ are
exactly the zeros of $z-c\rq{}$ and all the poles of $\displaystyle
\frac{z-c\rq{}}{z-c}$ are exactly the zeros of $z-c$.
\end{lemma}
The next lemma provides a bound on the chosen element $t$ in terms of the genus  of the field.
\begin{lemma}
\label{heigtoft}
There exists an element $t \in M$ satisfying conditions of Notation and Assumptions \ref{notation: pthpowers} such that 
\[
H(t) \leq \max(1, 2g-1).
\]
\end{lemma}
\begin{proof}
  If $g=0$, i.e.\ $M$ is a rational function field, the assertion of
  the lemma is clearly true.  So suppose $g >0$ and apply
  Lemma~\ref{Formulae}, Part (\ref{RR4}) with $d=2g-1$ to conclude
  that there is $x \in M$ whose height is $2g-1$.  By an argument
  similar to the one in Lemma \ref{cor:C(w)}, for any constant field
  large enough (and certainly for an infinite constant field) there
  exist constants $c, \tilde c$ such that $\displaystyle
  t=\frac{x-c}{x-\tilde c}$ does not have zeros or poles at primes
  ramifying in the extension $M/F(t)$.  Further, by Lemma
  \ref{zerosandpoles} we have that $H(t)=H(x-c)=H(x) \leq \max(1,
  2g-1)$.
\end{proof}
\begin{remark}
  While for the purposes of our arguments the bound on the height of
  $t$ is not important, since in the proofs below we only care about
  the fact that the height is fixed, it is useful to know that all the
  bounds in the paper are ultimately determined by the genus, which
  can serve as a measure of the Diophantine complexity of the field.
  Finally note that $k=[M:F(t)] =H(t)$ by Lemma \ref{FormulaOrd} and
  $i_G \leq k!$, so that all the constants occurring in the paper can
  be bound in terms of the genus of the field.
\end{remark}

The lemma below is perhaps the most important new technical part which allowed for the extension of earlier results.

\begin{lemma}%
\label{le:boundedheight}
Suppose $w, u,v \in M$ satisfy the following equations:
\begin{equation}%
\label{eq:1}
\left \{%
\begin{array}{c}%
w-t=v^{p^a} -v,\\%
\frac{1}{w} - \frac{1}{t}=u^{p^a} - u.%
\end{array}%
\right .%
\end{equation}%
If the divisor of $w$ is not a $p^a$-th power of another divisor, then $H(w) <C_3$.
\end{lemma}%
\begin{proof}%
  We assume that the divisor of $w$ is not a $p^a$-th power of another
  divisor in $M$ and obtain a bound on its height.  First of all note
  that all pole orders of $v^{p^a}-v$ and $u^{p^a} -u$ are 0 modulo
  $p^a$. Therefore, if for some prime $\rr$ we have that
  $\ord_{\rr}w \not =0$, then either $\ord_{\rr}w=\pm1$ or
  $\ord_{\rr}w \equiv 0 \mod p^a$.  Further, if $\ord_{\rr}w =-1$,
  then $\ord_{\rr}t=-1$ and $\ord_{\rr}v\geq0$.  Similarly, if
  $\ord_{\rr}w =1$, then $\ord_{\rr}t=1$ and $\ord_{\rr}u\geq0$.
  Given our assumption on the divisor of $w$, for at least one prime
  $\rr$ we have that $\ord_{\rr}w =1$ (implying that for at least one
  other prime the order is -1).  Thus
\[
(w)=\frac{\XX \Aa^{p^a}}{\YY\Bb^{p^a}},
\]
where $\XX, \YY, \Aa, \BB$ are pairwise relatively prime integral
divisors, multiplicity of all prime factors of $\XX$ and $\YY$ is one,
$\deg(\XX) <H(t) =C$, and $\deg(\YY) <H(t) =C$. Note that neither $\XX$ nor $\YY$ is a trivial divisor.  Further, without loss
of generality we can assume that $\deg(\XX) \leq \deg(\YY)$, and also
note that the pole divisor of $v$ is $\BB$ and of $u$ is $\Aa$. Now as
in Corollary \ref{cor:usingRR}, using the same notation, set
\[
w=\xi\frac{z_1^{p^a}}{z_2^{p^a}}, 
\]
where 
\[
H(\xi)\leq (p^a+1)(C+g+1),
\]
  rewrite rewrite the first equation of  \eqref{eq:1} as
\[%
\xi\frac{z_1^{p^a}}{z_2^{p^a}} -t = v^{p^a} -v,
\]%
\begin{equation}
\label{eq:2}
\xi z_1^{p^a}-tz_2^{p^a}=(vz_2)^{p^a}-(vz_2)z_2^{p^a-1},
\end{equation}
and set $s=vz_2$.  Note that if for some prime $\rr$ of $M$ we have that $\ord_{\rr}\Bb=h >0$, then 
\[
\ord_{\rr}s =0,
\]
\[
 \ord_{\rr}z_2 =h,
 \]
 \[
 \ord_{\rr}tz_2^{p^a} \geq hp^a-1\geq h(p^a-1),
 \]
 \[
  \ord_{\rr}\xi=0,
\]
\[
 \ord_{\rr}z_1=0.
\]
  We now rewrite \eqref{eq:2} to get
\[
\xi z_1^{p^a} -s^{p^a}=tz_2^{p^a} -sz_2^{p^a-1}.
\]
Observe that $\ord_{\rr}(tz_2^{p^a} -sz_2^{p^a-1}) \geq h(p^a-1) \geq 2h$,  and therefore $$\ord_{\rr}(\xi z_1^{p^a} -s^{p^a}) \geq h(p^a-1)\geq 2h.$$  

We now note that $\xi$ is not a $p$-th power in $M$ since at least one
zero or pole of $\xi$ has an order not divisible by $p$.  Thus the
global derivation with respect to $\xi$ is defined. (We are using our
assumption that the divisor of $w$ is not a $p^a$-th power in this
step.  Otherwise, $\XX$ and $\YY$ are trivial and $\xi$ is a constant,
and the derivation with respect to $\xi$ would not be defined.)  Taking
the derivative of $(\xi z_1^{p^a} -s^{p^a})$ with respect to $\xi$ we
see that it is equal to $z_1^{p^a}$ and thus $\ord_{\rr}(\xi z_1^{p^a}
-s^{p^a})\rq{} =0$.  At the same time we also have by Lemma
\ref{ramification}, that $\ord_{\rr}(\xi z_1^{p^a} -s^{p^a})\rq{} \geq
h(p^a-1) - 1 - d_{\xi}(\rr)$, implying that
\begin{equation}
\label{eq:hbound}
d_{\xi}(\rr) \geq h(p^a-1)-1>0.
\end{equation}
Thus $\rr$ belongs to a finite set of primes of size $$2g-2
+H(\xi) < 2g-2 + (p^a+1)(C+g+1)=C_1.$$ Using Lemma \ref{ramification}
again and \eqref{eq:hbound}, we can also obtain a bound on $h$:
 \[ 
 2g-1 + (p^a +1)(C+g+1) \geq d_{\xi}(\rr)+1 \geq h(p^a-1).
 \]
 Hence
 \[
 h <\frac{ 2g-1 + (p^a +1)(C+g+1) }{p^a-1}=C_2.
\]
Returning now to the structure of the divisor of $w$, we see that $$H(w) \leq \deg \XX + p^a\deg \BB \leq C +p^aC_1C_2.$$
\end{proof}  
The next lemma is a standard estimate of the height of the coefficients in a linear combination of basis elements in terms of the height of the linear combination itself.
\begin{lemma}
\label{height}
If $w \in M$ and $w=\sum_{i=1}^k A_i\omega_i$, where $A_i \in F(t)$, then $H(A_i) < k!k^kH_{\Omega}H(w)$.
\end{lemma}
\begin{proof}
  Consider a non-singular linear system $\sum_{i=1}^k
  A_i\sigma_j(\omega_i)=\sigma_j(w), j=1,\ldots,k$, where we consider 
  $A_1, \ldots, A_k$ as the unknowns.  Solving this system by
  Cramer\rq{}s Rule, and using the fact that the height of a
  sum/product is less or equal to the sum of heights, we can get an
  estimate on $H_{M^G}(A_i)$-- the height of $A_i$ in $M^G$.  More
  specifically, we get
\[
H_{M^G}(A_i)\leq 2k!k^k\max(H_{M^G}(\sigma_j(h), H_{\Omega}) \leq k!k^kH_{\Omega}H_{M^G}(w).
\]
Since for an element $z$ of $M$ we have that $H_{M^G}(z)=i_GH(z)$, we
cancel $i_G$ from both sides of the inequality and get the desired
result.
\end{proof}
The proposition below allows us to exploit fixed bounds on height.
Elsewhere, this proposition has been referred to as the Weak Vertical
Method.
\begin{proposition}[Slightly modified Theorem 10.1.1 of \cite{Sh34}]%
\label{prop:weak}
Suppose for some $w \in M$ with $H(w) < C_3$, for all primes $\mathfrak A$ of $\calV(w)$
 we have
\begin{equation}%
\label{eq:equiv}
w\equiv b(\mathfrak A) \mod \mathfrak A,
\end{equation}%
where $b(\mathfrak A) \in F$, and we interpret the equivalence as
saying that for any factor $\mathfrak c$ of $\mathfrak A$ in $M$ we
have that
\[
\ord_{\mathfrak c}(w-b(\mathfrak A)) \geq e(\mathfrak c/\mathfrak A),
\]
 the ramification degree of $\mathfrak c$ over
$\mathfrak A$. Then $w \in F(t)$.
\end{proposition}%
\begin{proof}%
  First of all we note that by Part \eqref{it:3} of the description
  of $\calV(w)$ in Notation and Assumption \ref{notation: pthpowers},
  any element $z \in M$ integral with respect to $\mathfrak A \in
  \calV(w)$, i.e.\ integral with respect to every factor of $\mathfrak
  A$ in $M$, can be written as
\[
z = \sum_{i=1}^k f_i\omega_i,
\]
 where for all $i =
1,\ldots,n$, we have that $f_i \in F(t)$ and $f_i$ is integral at
$\mathfrak A$.  Thus,  we now write
$w=\sum_{i=1}^kA_i\omega_i$, where $A_i \in F(t)$.  Observe that for
all $\mathfrak A \in \calV(w)$, we have that $w-b(\mathfrak A)$ is
equivalent to zero modulo every prime ${\mathfrak A}$ of $\calV(w)$.  At
the same time 
\[
w- b(\mathfrak A)= A_1 - b(\mathfrak A)+ A_2\omega_2 +\dots + A_k\omega_k.
\]
For each prime ${\mathfrak A}$ of $\calV(w)$, let $B(\mathfrak A) \in
F(t)$ be such that $\ord_{\mathfrak A}B(\mathfrak A) =1$.  (Such a
$B(\mathfrak A))$ exists by the Weak Approximation Theorem.)  Note
that $\displaystyle z=\frac{w-b(\mathfrak A)}{B(\mathfrak A)}$ is
integral at $\mathfrak A$ and thus $z = \sum_{i=1}^k f_i(\Aa)
\omega_i$, where $f_i(\Aa)$ are elements of $F(t)$ integral at
$\mathfrak A$.  Furthermore,
\[
A_1 - b(\mathfrak A)+ A_2\omega_2 + \dots + A_k\omega_k=w-b(\mathfrak
A) = B(\mathfrak A)z=\sum_{i=1}^k B(\mathfrak A)f_i(\Aa)\omega_i.
\]
Thus, for $i=2,\ldots,k$ for all ${\mathfrak A}$ in $\calV(w)$ we have
that $A_i = B(\mathfrak A)f_i(\Aa)$, implying that for $i=2,\ldots,k$,
for all ${\mathfrak A}$ in $\calV(w)$ we have that $\ord_{\mathfrak
  A}A_i >0$, implying that $H(A_i) >C_4$ or $A_i=0$.  The last
inequality contradicts our assumption on $H(w)$ and Lemma
\ref{height}. Therefore for $i=2,\ldots,k$ we have that $A_i=0$ and
thus $w \in F(t)$.
\end{proof}%
We now apply the Weak Vertical Method to our situation.
\begin{corollary}%
\label{finalpart1}%
Suppose that for some $w \in M$ with $H(w) < C_3$, and $C_4+e$
quadruples $(b, c,b\rq{},c\rq{}) \in F^4$ with $b\not=b\rq{}$, we have
a solution $u_b$ in $M$ to
\begin{equation}
\label{eq:getdown}%
\frac{w-c\rq{}}{w-c}-\frac{t-b\rq{}}{t-b}=u_b^{p^a}-u_b.
\end{equation}%
In this case,  $w \in F(t)$.  
\end{corollary}%
\begin{proof}%
  Let $b \in F$ be such that it occurred as the first element in one
  of the quadruples above, so that the $F(t)$-prime $\mathfrak P_b$
  corresponding $t-b$ is unramified in the extension $M/F(t)$, and let
  $\pp_b$ be any factor of $\mathfrak P_b$ in $M$. (We have at least
  $C_4$ such elements $b$.)  Since $\pp_b$ is unramified,
  Lemma~\ref{zerosandpoles} implies that $\displaystyle
  \ord_{\pp_b}(t-b)=-\ord_{\pp_b}\frac{t-b\rq{}}{t-b}=1$.  At the same
  time, for any pole $\qq_b$ of $u_b$ in $M$ we have that
  $\displaystyle \ord_{\qq_b}(u_b^p-u_b) \equiv 0 \mod p^a$ as
  above. Thus, $c \not = c'$ and, by Lemma \ref{zerosandpoles} again,
  we have that $ \displaystyle
  -\ord_{\pp_b}\frac{w-c\rq{}}{w-c}=\ord_{\pp_b}(w-c) >0$.  In other
  words, for $C_4$ pairs $(b,c) \in F^2$ we have that $w \equiv c \mod
  \mathfrak P_b$, where $\mathfrak P_b$ is, as above, the zero divisor
  in $M$ and $F(t)$ of $t-b$. Now the assertion of the corollary
  follows from Proposition \ref{prop:weak}.
\end{proof}%

%
The next proposition gives equations that let us conclude that an element $w$ is a $p^s$-th
power of $t$ provided that we know that $w$ is in the rational
function field $F(t)$.
\begin{proposition}[\cite{Sh34}, Lemma 8.3.3, Corollary 8.3.4 and
  \cite{Eis}, Lemma 3.4] %
\label{prop:version1}  
 Suppose for some element $w \in F(t)$, having no poles or zeros at the
primes ramifying in the extension $M/F(t)$, there exist $u, v \in M$
such that the following system is satisfied:
\begin{equation}%
\label{eq:alreadydown0}%
\left \{ \begin{array}{c}%
\displaystyle \frac{1}{w}-\frac{1}{t}=u^{p^a}-u\\
w-t=v^{p^a}-v.%
\end{array} \right . %
\end{equation}%
Then for some $s \in \Z_{\geq 0}$ we have that $w=t^{p^{as}}$.
\end{proposition}%
\comment{
\begin{proof}%
  First of all note that all the poles of $v^{p^a}-v$ and $u^{p^a}-u$
  in $M$ are of orders divisible by $p^a$.  Since zeros and poles of
  $t$ are of orders equal to $\pm 1$ (see Notation \ref{notation:
    pthpowers}), we must conclude from \eqref{eq:alreadydown0} that
  the divisor of $w$ in $M$ is of the form
  $\Uu^{p^a}\Pp^{b_1}\Qq^{b_2}$. Indeed, let $\ttt$ be a prime which
  is not a factor of $\mathfrak P$ or $\mathfrak Q$. Without loss of
  generality assume $\ttt$ is a pole of $w$. Then, since
  $\ord_{\mte{t}}t = 0$,
\[%
0 > \ord_{\mte{t}}w = \ord_{\mte{t}}(w-t) = \ord_{\mte{t}}(v^{p^a}-v)
\equiv 0 \mbox{ modulo }p^a.%
 \]%
 Now let $\qq$ be a factor of $\mathfrak Q$ and consider the order at
 $\qq$. We have%
\[%
\ord_{\mte{q}}(w-t) =  \ord_{\mte{q}}(v^{p^a}-v).%
\]%
Therefore, either $\ord_{\mte{q}}w < -1 $ and $\ord_{\mte{q}}w \equiv
0 \mod p^a$, or $\ord_{\mte{q}}w=\ord_{\mte{q}}t=-1$.  Since $w \in
F(t)$ and no factor of $\mathfrak Q$ is ramified in the extension
$M/F(t)$, the order of $w$ at all the factors of $\mathfrak Q$ is the
same.  Thus, either $b_2=-1$ or $b_2 \equiv 0 \mod p^a$.  Similarly,
for any prime factor $\pp$ of $\mathfrak P$ either $\ord_{\mte{p}}w >
1 $ and $\ord_{\mte{p}}w \equiv 0 \mod p^a$, or $\ord_{\mte{p}}w
=\ord_{\mte{p}}t=1$.  Thus, $b_1=1$ or $b_1 \equiv 0 \mod p^a$.  Now
we will examine the divisor of $w$ in $F(t)$, where $\Pp$
and $\Qq$ are primes of $F(t)$.  Since, by assumption $w$ does not
have any poles or zeros at primes ramifying in the extension $M/F(t)$,
for any $F(t)$-prime $\mathfrak t \not \in \{\Pp, \Qq\}$ we must have
$\ord_{\mathfrak t}w\equiv 0 \mod p^a$.  Thus in $F(t)$ the divisor of
$w$ is of the form $\mathfrak D^{p^a}\Pp^{b_1}\Qq^{b_2}$.  Since the
degree of this divisor must be 0, we deduce that either
$|b_1|=|b_2|=1$ or $b_1 \equiv b_2 \equiv 0 \mod p^a$.  We will
consider the two cases separately.

If we first assume that $|b_1|=|b_2|=1$, then we can deduce that
$\ord_{\pp}u \geq 0$ for any factor of $\Pp, \ord_{\qq}v\geq 0$ for
any factor $\qq$ of $\Qq$, and also $wt^{-1}$ is a $p^a$-th power in
$F(t)$.  The first assertion follows from the fact that
$\ord_{\qq}t=\ord_{\qq}w=-1$ and therefore $\ord_{\qq}(w-t) \geq -1$
but cannot be equal to $-1$.  The second assertion can be derived in a
similar fashion by looking at $\frac{1}{w} - \frac{1}{t}$.  The third
assertion follows from the fact that in $F(t)$ every zero degree
divisor is principal and the constant field is perfect.  Thus, the
second equation of \eqref{eq:alreadydown0} can be rewritten as
\begin{equation}
\label{eq:003}
t(f-1)^{p^a} = v^{p^a} - v,
\end{equation}
where $f \in F(t)$.  Since $f-1$ is a rational function in $t$, we can
further rewrite (\ref{eq:003}) as%
\begin{equation}%
\label{eq:without}%
t(f_1^{p^a}/f_2^{p^a}) = v^{p^a} - v,%
 \end{equation}%
 where $f_1,f_2$ are relatively prime polynomials in $t$ over
 $F$. From this equation it is clear that any valuation that is a pole
 of $v$, is either a pole of $t$ or a zero of $f_2$.  But by the
 discussion above, $v$ does not have any poles at factors of the pole
 divisor of $t$.  So all the poles of $v$ are zeros of $f_2$.
 Further, the absolute value of the order of any pole of $v$ at any
 valuation which is a zero of $f_2$, must be the same as the order of
 $f_2$ at this valuation. Therefore, $s = f_2v$ will have poles only
 at the valuations which are poles of $t$.  (These poles are from
 $f_2$.) We can rewrite (\ref{eq:without}) in the form
\[%
-tf_1^{p^a} + s^{p^a} = sf_2^{p^a-1}.%
\]%
Let $\cc$ be a zero of $f_2$. Then, since $f_2$ is a polynomial in
$t$, we have that $\cc$ is not a pole of $t$. Since $p^a-1 \geq 2$ and
$s$ is integral over $F[t]$, we have that $\mbox{ord}_{\cc}(s^{p^a} -
tf_1^{p^a}) \geq 2$.  Now since $t$ is not a $p$-th power, as above,
the global derivation with respect to $t$ is defined and
\[%
d( -tf_1^{p^a} + s^{p^a}) /dt = -f_1^{p^a}.
\]%
Since, by assumption, $f_2$ does not have any zeros at valuations
ramifying in the extension $M/F(t)$,  Corollary~\ref{cornoram} implies
\[%
\mbox{ord}_{\mbox{\te c}}(-f_1^{p^a})= \mbox{ord}_{\mbox{\te c}}(d(
-tf_1^{p^a} + s^{p^a}) /dt) >0.%
\]%
Thus, $f_1$ has a zero at $\cc$. But $f_1$ and $f_2$ are supposed to
be relatively prime polynomials. Hence, $f_2$ does not have any zeros,
and therefore $v$ has no poles making it a constant.  Similarly, $u$
is a constant, and thus $w=t+c_1$ and $\frac{1}{w}=\frac{1}{t} +c_2$,
where $c_1, c_2 \in M$.  These two equalities can hold if and only if
both constants are zero.

Next consider the case of $b_1 \equiv b_2 \equiv 0 \mod p^a$. In this
case, since $w$ does not have zeroes or poles ramifying in the
extension $M/F(t)$, the divisor of $w$ in $F(t)$ is a $p^a$-th power
of another divisor, and as above this implies that $w = \tilde
w^{p^a}$ for some $\tilde w \in F(t)$.  So  if $w \not
=t$, we have that $w = \tilde w^{p^a}$ for some $\tilde w \in F(t)$.
Assuming $w \not = t$ we do the following.  Set
$\tilde{v}=v-\tilde{w}, \tilde{u}=u-\tilde{w}^{-1}$ and observe that
the following equations hold.
 \[%
\tilde{w} - t =(v - \tilde{w})^{p^a} - (v- \tilde{w})=\tilde v^{p^a}-\tilde{v}
\]%
\[%
 \tilde{w}^{-1} - t^{-1} = (u-\tilde{w}^{-1})^{p^a} - (u-\tilde{w}^{-1}) =\tilde u^{p^a} - \tilde{u}%
 \]%

 Since we can repeat this process only finitely many times, we conclude
 that $w=t^{p^{as}}$ for some $s \in \Z_{\geq0}$.

\end{proof}%
}
In general we do not know if $w$ has all of its poles and zeros at
primes not ramifying in the extension $M/F(t)$.  Therefore we might
have to replace $w$ by $\displaystyle \frac{w-b}{w-b'}$, where $b, b'
\in F_0$. (Recall that $F_0$ is the algebraic closure of $\F_p$ in $F$.)  The proposition below carries out this construction.

\begin{proposition}%
\label{prop:below}%
 Let $w \in F(t)$, assume the system \eqref{eq:alreadydown0} holds, and for all $i\not = j \in \{1,\ldots, C_5\}$, for some $b \in V_i, b\rq{}\in V_j$ there exist
$u_{i,j,b,b\rq{}}, v_{i,j,b,b\rq{}} \in M$ such that
\begin{equation}%
\label{eq:alreadydown1}%
\left \{ \begin{array}{c}%
\displaystyle \frac{w-b}{w-b\rq{}}-\frac{t-c_i}{t-c_j}=u_{i,j,b,b\rq{}}^{p^a}-u_{i,j,b,b\rq{}}\\
\displaystyle \frac{w-b\rq{}}{w-b}-\frac{t-c_j}{t-c_i}  =v_{i,j,b,b\rq{}}^{p^a}-v_{i,j,b,b\rq{}}%
\end{array} \right . %
\end{equation}%

In this case for some $s \in \Z_{\geq 0}$ we have that $w=t^{p^{as}}$.

\end{proposition}%
\begin{proof}%
  First of all, by Lemma \ref{cor:C(w)} we
  conclude that for some $c_i, c_j$ and $b \in V_i, c \in V_j$, we
  have that $t-c_i, t-c_j, w-b, w-c$ do not have zeros at any prime
  ramifying in the $M/F(t)$ and therefore
  $\displaystyle \frac{t-c_i}{t-c_j}, \frac{w-b}{w-b'} \in F(t)$ do not
  have zeros or poles at any primes ramifying in the extension
  $M/F(t)$.   It follows that all the zeros and poles of  $\displaystyle \frac{t-c_i}{t-c_j}$ are simple, since they are simple in $F(t)$.  Now applying Proposition \ref{prop:version1} we conclude
  that 
\begin{equation}
\label{eq:pr}
\frac{w-b}{w-b'} = \left(  \frac{t-c_i}{t-c_j}\right)^{p^{as}}
\end{equation}
for some $s \geq 0$.    From \eqref{eq:pr} we deduce
\begin{equation}
\label{eq:middle}
1 + \frac{b'-b}{w-b'}=1+ \frac{c^{p^{as}}_j-c^{p^{as}}_i}{t^{p^{as}}-c_j^{p^{as}}}.
\end{equation}
Since we know from the second equation of \eqref{eq:alreadydown0} that
$t$ and $w$ have a common zero, considering the equation above modulo
this prime gives us
\[
\frac{b'-b}{b'}= \frac{c^{p^{as}}_j-c^{p^{as}}_i}{c_j^{p^{as}}},
\]
or
\[
\frac{b}{b'}= \frac{c^{p^{as}}_i}{c_j^{p^{as}}}.
\]
Thus, from \eqref{eq:middle}, for some $r \in \Z_{\geq 0}$ we have that
\[
w=b'+ \frac{b'-b}{c^{p^{as}}_j-c^{p^{as}}_i}(t^{p^{as}}-c_j^{p^{as}})= b' +\frac{b'}{c^{p^{as}}_j}(t^{p^{as}}-c_j^{p^{as}})=\frac{b'}{c^{p^{as}}_j}t^{p^{as}}=\frac{c_j^{p^r}}{c^{p^{as}}_j}t^{p^{as}}=c_j^{p^r-p^{as}}t^{p^{as}}.
\]
In a similar fashion we conclude that for some $m \in \Z_{\geq 0}$ we have that $w=c_i^{p^m-p^{as}}t^{p^{as}}$ and hence $c_j^{p^r-p^{as}}  = c_i^{p^m-p^{as}}$.  By assumption on elements of $C(F)$ we must have $p^r-p^{as}=p^m-p^{as}=0$ and $w=t^{p^{as}}$.
\end{proof}%

We will now prepare for the case when we cannot conclude right away
that $w$ is of bounded height and use the Weak Vertical Method to see
that it is in the fixed rational subfield.  In this case by Lemma
\ref{le:boundedheight}, the divisor of $w$ is a $p^{as}$-th power of
another divisor.  The three lemmas below take advantage of this fact
to conclude that under certain conditions $w$ is a $p^a$-th power of
another field element. The proofs for all three lemmas can be found in \cite{Sh34}.


\begin{lemma}[\cite{Sh34}, Lemma 8.2.4]
\label{le:coefficients}%
Let $M/G$ be a finite separable extension of fields of positive
characteristic $p$. Let $\alpha \in M$ be such that for some positive
integer $a$, all the coefficients of its monic irreducible polynomial
over $G$ are $p^a$-th powers in $G$. Then $\alpha$ is a $p^a$-th power
in $M$.
\end{lemma}
\comment{
\begin{proof}
  Let $a_0^{p^a} + \dots + a_{m-1}^{p^a}T^{m-1} + T^{m}$ be the monic
  irreducible polynomial of $\alpha$ over $G$.  Let $\beta$ be the
  element of the algebraic closure of $M$ such that $\beta^{p^a} =
  \alpha$. Then $\beta$ is of degree at most $m$ over $G$. On the
  other hand, $G(\alpha) \subseteq G(\beta)$. Therefore, $G(\alpha) =
  G(\beta)$.%
\end{proof}
}
\begin{lemma}[\cite{Sh34}, Lemma 8.2.5]
\label{lemma:norms}%
Let $M/G$ be a finite separable extension of fields of positive
characteristic $p$. Let $[M:G] =n$. Let $r$ be a positive integer.
Let $x \in M$ be such that $M= G(x)$ and for distinct $b_0,\ldots,b_n
\in G$ we have that {\bf N}$_{M/G}(b_i^{p^r}-x) = y_i^{p^r}$. Then $x$
is a $p^r$-th power in $M$.
\end{lemma}
\comment{
\begin{proof}
  Let $U(T) = A_0 + A_1T + \dots +A_{n-1}T^{n-1} + T^n$ be the monic
  irreducible polynomial of $x$ over $G$. Then for $i = 0, \ldots,n$
  we have that $U(b_i^{p^r}) = y_i^{p^r}$. Further, we have the
  following linear system of equations:
\[%
 \left (%
  \begin{array}{ccccc}%
 1&b_0^{p^r}&\ldots&b_0^{p^r(n-1)}&b_0^{p^rn}\\ \vdots&\vdots&\vdots&\vdots&\vdots\\%
1&b_n^{p^r}&\ldots&b_n^{p^r(n-1)}&b_n^{p^rn}\\ \end{array} \right ) \left ( \begin{array}{c} A_0\\ \vdots\\ 1\\%
\end{array}
\right ) = %
\left (%
\begin{array}{c}%
y_0^{p^r}\\ \vdots\\ y_n^{p^r}\\
\end{array}%
\right )%
\]%
Using Cramer's rule to solve the system, it is not hard to conclude
that for $i = 0,\ldots,n$ it is the case that $A_i$ is a $p^r$-th
power in $G$. Then, by Lemma \ref{le:coefficients}, $x$ is a $p^r$-th
power in $M$.
\end{proof}
}
\begin{lemma}[\cite{Sh34}, Lemma 8.4.1]
\label{lemma:divpth}%
Let $M$ be a function field over a perfect field of constants $L$ and
let $t \in M$ be such that $M/L(t)$ is a finite and separable
extension of degree $n$. Let $m$ be a positive integer. Let $v \in M$ and assume
that for some distinct $b_0 = 0, b_1,\ldots,b_n \in L$, the divisor of
$v + b_0, \ldots, v + b_n$ is a $p^m$-th power of some other divisor
of $M$. Then, assuming for all $i$ we have that $v + b_i$ does not
have any zeros or poles at any prime ramifying in the extension
$M/L(t)$, it is the case that $v$ is a $p^m$-th power in $M$.
\end{lemma}
\comment{
\begin{proof}%
  First assume $v \in L(t)$. Since $v + b_i$ does not have any zeros
  or poles at primes ramifying in the extension $M/L(t)$, the divisor
  of $v + b_i$ in $L(t)$ is a $p^m$-th power of another $L(t)$
  divisor. Since in $L(t)$ every zero degree divisor is principal and
  the constant field is perfect, $v$ is a $p^m$-th power in $L(t)$ and
  therefore in $M$. Next assume $v \not \in L(t)$. Note that no zero
  or pole of $v+ b_i$ is at any valuation ramifying in the extension
  $M/L(t,v)$. Hence, in $L(t,v)$ the divisor of $v+b_i$ is also a
  $p^m$-th power of another divisor. Finally note that {\bf
    N}$_{L(t,v)/L(t)}(v + b_i)$ will be a $p^m$-th power in $L(t)$ and
  apply Lemma \ref{lemma:norms}.\\%
\end{proof}
}
We are now ready to put all the parts together.

\begin{proposition}
\label{theend}
Suppose for some $w \in M$ we have that \eqref{eq:alreadydown0} and
\eqref{eq:alreadydown1} hold with all the variables taking values in
$M$.  In this case $w=t^{p^{as}}$ for some positive integer $s$.

\end{proposition}

\begin{proof}
Suppose \eqref{eq:alreadydown0}  and \eqref{eq:alreadydown1}  hold.  We need to consider two cases:  

{\it Case 1:} For one pair $c_i, c_j$ with $t-c_i$ and $t-c_j$
corresponding to primes that are not ramified (over $F(t)$), we have
that the divisor of $\displaystyle \frac{w-b}{w-b'}$ is not a
$p^a$-th power of another divisor in $M$.
In this case applying Lemma \ref{le:boundedheight} we conclude that
$\displaystyle  H\left (\frac{w-b}{w-b'}\right )=H(w) < C_3$.  (The
equality of heights follows from Lemma \ref{zerosandpoles}.)  Now by
Corollary \ref{finalpart1} we conclude that $w \in F(t)$.  
Applying Proposition \ref{prop:below}, we finally conclude that $w=t^{p^{as}}$ for some $s \in \Z_{\geq 0}$.\\

{\it Case 2:} For all values of $c_i \not =c_j$ such that the
$F(t)$-primes corresponding to $t-c_i$ and to $t-c_j$ are not
ramified, we have that that the divisors of $\displaystyle
\frac{w-b'}{w-b}$ are $p^a$-th powers of other divisors. (We remind
the reader that $b \in V_i, b' \in V_j$.)  In this case, the divisor
of $\displaystyle 1+\frac{b-b'}{w-b}$ is a $p^a$-th power of another
divisor and therefore the divisor of
\[
\frac{1}{b-b'}+\frac{1}{w-b'}:=d_{i,j} +\frac{1}{w_j}
\]
is a $p^a$-th power of another divisor for all such $i\not =j$. This
follows since
$\displaystyle \frac{w-b'}{w-b}$ and $d_{i,j} +\frac{1}{w_j}$ differ
by a constant factor only, and therefore have the same divisor in $M$.
By Lemma \ref{cor:C(w)} we have that $|C_t \cap C_w| > 2k$.  Thus, to
summarize the discussion above, for a fixed value of $j$ with $c_j \in
C_t \cap C_w$ and at least $k+1$ values of $i \not =j$ with $c_i \in C_t
\cap C_w$, we have that $\displaystyle d_{i,j} + \frac{1}{w_j}$ does
not have a pole or a zero at a prime ramified over $F(t)$. Also, for any
pair $i_1 \not = i_2$ we have $d_{i_1,j} \not = d_{i_2,j}$, and the
divisor of each $\displaystyle d_{i,j} + \frac{1}{w_j}$ is a $p^a$-th
power of another divisor.  Hence by Lemma \ref{lemma:divpth} we
conclude that $w_j$ for this $j$ is a $p^a$-th power in $M$, and thus
$w$ is a $p^a$-th power in $M$.

At this point we can take the ``$p^a$-th root\rq{}\rq{} of our
equations as was done in Proposition \ref{prop:below} and again ask,
this time for the ``new'' $w$, whether the divisor of $\displaystyle
\frac{w-b'}{w-b}$ is not a $p^a$-th power of another divisor for some
$i, j$ with $c_i, c_j \in C_t \cap C_w$.

Since our ``$p^a$-th root descent\rq{}\rq{} cannot go on indefinitely,
at some step we will conclude that the divisor of $\displaystyle
\frac{w-b'}{w-b}$ is not a $p^a$-th power of another divisor for some
$i, j$ with $c_i, c_j \in C_t$.  When this happens, we will follow the
{\it Case 1} argument to reach the desired conclusion.

\end{proof}
The results in Sections~\ref{simple}, \ref{subsec:cansatisfy} and
\ref{all} are only slight modifications of known results going back
in some form to \cite{Ph5}.  We
include these results and some of the proofs for the convenience of
the reader.
\subsection{Defining $p$-th powers of elements with simple zeros and poles}
\label{simple}
  In this section we need additional notation listed below.
\begin{notationassumption}
$\left.\right.$
\begin{itemize}
\item For $s \in \Z_{\geq 0}, i, l \in \{1,\ldots,C_5\}, j_i \in \{1,\ldots, r_i\}, j_l \in \{1,\ldots, r_l\}$,  $z = -1,1$,
$m = 0,1$, $u, v, \mu_{i,j_i,l,j_l,z,m}$, $\lambda_1, \lambda_{-1},\sigma_{i,j_i,l,j_l}   \in M$, let
\[%
D(s,i,j_i,l,z,m, j_l,u, v, \mu_{i,j_i,l,j_l,z,m}, \sigma_{i,j_i,l,j_l}, \lambda_1, \lambda_{-1}) %
\]%
be the following system of equations:
\begin{equation}
\label{eq:duf2.1}
u_{i,k} = \frac{u + c_i}{u + c_l},
\end{equation}
\begin{equation}
\label{eq:2.2}
v_{i,j_i,l,j_l} = \frac{v+d_{i,j_i}}{v + d_{l,j_l}},
\end{equation}
\begin{equation}
\label{eq:2.3}
v_{i,j_i,l,j_l}^{2z}t^{mp^{as}} - u_{i,l}^{2z}t^m = \mu_{i,j_i,l,j_l,z,m}^{p^{as}} - \mu_{i,j_i,l,j_l,z,m},
\end{equation}
\begin{equation}
\label{eq:2.4}
v_{i,j_i,l,j_l} - u_{i,k} = \sigma_{i,j_i,l,j_l}^{p^a} - \sigma_{i,j_i,l,j_l},%
\end{equation}%
\begin{equation}
\label{eq:2.5}
v - u =\lambda_1^{p^a} - \lambda_1
\end{equation}
\begin{equation}
\label{eq:2.5.1}
v^{-1} - u^{-1}=\lambda_{-1}^{p^a} - \lambda_{-1}
\end{equation}

\item Let $j,r,s \in \Z_{\geq 0}$, $u,\tilde{u},v, \tilde{v}, x, y \in
  M$.  Let $E(u,\tilde{u},v, \tilde{v}, x, y, j,r,s)$ denote the
  following system of equations
\begin{equation}%
\label{eq:final1}%
v=u^{p^r}%
\end{equation}%
\begin{equation}%
\label{eq:final2}%
\tilde{v}=\tilde{u}^{p^j}%
\end{equation}%
\begin{equation}%
\label{eq:final3}%
u=\frac{x^p+t}{x^p-t}%
\end{equation}
\begin{equation}%
\label{eq:final4}%
\tilde{u}=\frac{x^p+t^{-1}}{x^p-t^{-1}}%
\end{equation}
\begin{equation}%
\label{eq:final5}%
v=\frac{y^p+t^{p^s}}{y^p-t^{p^s}}%
\end{equation}%
\begin{equation}%
\label{eq:final6}%
\tilde{v}=\frac{y^p+t^{-p^s}}{y^p-t^{-p^s}}%
\end{equation}%
\item Let $j,r,s \in \Z_{\geq 0}$, $u,\tilde{u},v, \tilde{v}, x, y \in
  M$, and let $E2(u,\tilde{u},v, \tilde{v}, x, y, j,r,s)$ denote the
  following system of equations
\begin{equation}%
\label{eq:final12}%
v=u^{2^r}%
\end{equation}%
\begin{equation}%
\label{eq:final22}%
\tilde{v}=\tilde{u}^{2^j}%
\end{equation}%
\begin{equation}%
\label{eq:final32}%
u=\frac{x^2+t^2+t}{x^2+t}%
\end{equation}
\begin{equation}%
\label{eq:final42}%
\tilde{u}=\frac{x^2+t^{-2}+t^{-1}}{x^2+t^{-1}}%
\end{equation}
\begin{equation}%
\label{eq:final52}%
v=\frac{y^2+t^{2^{s+1}}+ t^{2^s}}{y^2+t^{2^s}}%
\end{equation}%
\begin{equation}%
\label{eq:final62}%
\tilde{v}=\frac{y^2+t^{-2^{s+1}}+t^{-2^s}}{y^2+t^{-2^s}}%
\end{equation}%
 \end{itemize}%
\end{notationassumption}

 We start with a way to produce elements with simple zeros and poles.
 
\begin{lemma}[\cite{Sh13}, Lemma 4.5 or \cite{Sh34}, Lemma 8.4.2] %
\label{le:simplezero}%
 Let $p >2$. Let $x \in M$. Let $\displaystyle u = \frac{x^p +  t}{x^p - t}$. Let $b \in F, b \not = \pm1$. Then all zeros and
poles of $u^{\pm1}+b$ are simple except possibly for zeros or poles of $t$ or
at primes ramifying in the extension $M/F(t)$.
\end{lemma}
\begin{proof}%
  It is enough to show that the proposition holds for $u$.  The
  argument for $u^{-1}$ follows by symmetry. First of all we remind the reader that 
  the global derivation with respect to $t$ is defined over $M$, and
  the derivative follows the usual rules. So consider
\[%
\frac{d(u+b)}{dt} = \frac{2x^p}{(x^p-t)^2}.%
\]%
If $\ttt$ is a prime of $M$ such that $\ttt$ does not ramify in the
extension $M/F(t)$ and is not a pole or zero of $t$, then 
Corollary \ref{cornoram} implies that
\[%
\ord_{\mte{t}}(u+b)=\ord_{\ttt}\frac{(1+b)x^p+(1-b)t}{x^p-t} > 1%
\]%
if and only if $\ttt$ is a common zero of $u+b$ and
$\displaystyle \frac{d(u+b)}{dt}$. If $\displaystyle \ord_{\ttt}\frac{2x^p}{(x^p-t)^2} >0$, then
$\ttt$ is either a zero of $x$ or a pole of $x^p-t$. Any zero of $x$,
which is not a zero of $t$, is not a zero of $u+b$ for $b \not
=1$. Furthermore, any pole of $x$ is also not a zero of $u+b$. Thus all
zeros of $u+b$ at primes not ramifying in the extension $M/F(t)$
and different from poles and zeros of $t$ are simple. Next we note that poles
of $u+b$ are zeros of $u^{-1}$.  Further
\[%
\frac{du^{-1}}{dt}=  \frac{-2x^p}{(x^p+t)^2},%
\]%
and by a similar argument, $u^{-1}$ and $\frac{du^{-1}}{dt}$ do not
have any common zeros at any primes not ramifying in the extension
$M/F(t)$ and not being  poles or  zeros of $t$.
\end{proof}
The following lemma (which we state without a proof) deals with the case of $p=2$.  
\begin{lemma}[\cite{Eis}, Lemma 3.8]%
\label{le:simplezero2}%
Let $p=2$ and $x \in M$. Let $\displaystyle u = \frac{x^2 +t^2+ t}{x^2 + t}$. Let $b \in F, b
\not =1$. In this case all zeros and poles of $u+b$ are simple except possibly
for zeros or poles of $t$ or at primes ramifying in the extension
$M/F(t)$.
\end{lemma}

The lemma below is a technical result we need to define the $p^a$-th powers of elements with simple zeros and poles.
 \begin{lemma}[This lemma is slightly modified from \cite{Sh34}, Lemma 8.2.11]
  \label{lemma:eqzero}%
  Let $\sigma, \mu \in M$. Assume that all the primes that are poles of $\sigma$ or
  $\mu$ do not ramify in the extension $M/F(t)$. Further, assume  the following equality holds.
\begin{equation}%
\label{eq:1.3}%
t(\sigma^{p^a} - \sigma) = \mu^{p^a} - \mu.%
\end{equation}%
In this case $\sigma^{p^a} - \sigma = \mu^{p^a} -\mu = 0$. (Here we
remind the reader that by assumption, the primes occurring in the
divisor of $t$ do not ramify in $M/F(t)$.)
\end{lemma}
\begin{proof}%
  Let $\Aa, \Bb$ be integral divisors of $M$, relatively prime to
  each other and to ${\mathfrak P}=\prod_i\pp_i$ and ${\mathfrak    Q}=\prod_i\qq_i$ (in other words, no prime occurring in $\Aa$ or $\Bb$ occurs in the divisor of $t$), and such that the divisor of $\sigma$ is of the form
  $\displaystyle \frac{\Aa}{\Bb}\prod_i {\pp_i}^{n_i}\prod_i{\qq_i}^{k_i}$, where   $n_i,k_i$ are integers for all $i$.  It is not hard   to see that for some integral divisor ${\mathfrak C}$, relatively
  prime to $\Bb,\Pp$, and $\Qq$, some  integers $a_i,b_i$, the divisor
  of $\mu$ is of the form $\displaystyle \frac{\Cc}{\Bb}\prod_i
  \pp_i^{a_i} \prod_i\qq_i^{b_i}$. Indeed, if $\ttt$ is a pole of $\mu$  such that $\ttt$ does not divide $\Pp$ or $\Qq$, then
\[%
0 > p^a\ord_{\mte{t}}\mu = \ord_{\mte{t}}(\mu^{p^a} - \mu) = \ord_{\mte t}(t(\sigma^{p^a} - \sigma)) =
\ord_{\mte{t}}(\sigma^{p^a} - \sigma) = p^a\ord_{\mte{t}}\sigma.%
\]%
 Conversely, if $\ttt$ is a  pole of $\sigma$ such that $\ttt$ does not divide $\Pp$ or $\Qq$, then%
\[%
0 > p^a\ord_{\ttt}\sigma = \ord_{\ttt}(\sigma^{p^a} - \sigma) = \ord_{\ttt}(t(\sigma^{p^a} - \sigma)) =
\ord_{\ttt}(\mu^{p^a} - \mu) = p^a\ord_{\ttt}\mu. %
\]%
Further we can also deduce that for each $\pp_i$ we have that
$\ord_{\pp_i}\sigma \geq 0$ and $\ord_{\pp_i}\mu \geq 0$.  To see that
this is the case suppose $\ord_{\pp_i}\sigma <0$ and conclude that
\begin{align}
\label{eq:polesigma}
\ord_{\pp_i}t(\sigma^{p^a} - \sigma) <0, \mbox{ and }\\
\ord_{\pp_i}t(\sigma^{p^a} - \sigma) \not \equiv 0 \mod p.
\end{align}
At the same time \eqref{eq:polesigma} implies that
\begin{align}
\label{eq:polemu}
\ord_{\pp_i}(\mu^{p^a} - \mu) <0, \mbox{ and }\\
\ord_{\pp_i}(\mu^{p^a} - \mu) \equiv 0 \mod p.
\end{align}
Therefore assuming $\ord_{\pp_i}\sigma <0$ leads to a contradiction.
Similarly, if $\ord_{\pp_i}\mu <0$ then \eqref{eq:polemu} and
\eqref{eq:polesigma} hold and we again obtain a contradiction.
Assuming that $\ord_{\qq_i}\sigma <0$, $\ord_{\qq_i}\mu <0$ results
in a contradiction of a similar type.  Thus, we can assume that $a_i, b_i, n_i, k_i \geq 0$ for all $i$.

By the Strong Approximation Theorem there exists $b \in M^{\times}$ such that
the divisor of $b$ is of the form $\Bb\Dd/\qq_1^c$, where $\Dd$ is an
integral divisor relatively prime to $\Aa, \Bb, \Cc, \Pp, \Qq$, and $c$ is
a positive integer. Then $b\sigma = s_1, b\mu = s_2$, where $s_1, s_2$
are integral over $F[t]$ and have zero divisors relatively prime to
$\Bb$. Indeed, consider the divisors of $s_1=b\sigma$: %
\[%
\frac{\mathfrak {BD}}{\qq_1^c}\frac{\mathfrak A}{\mathfrak B}\prod_i \pp_i^{n_i}\prod_j\qq_j^{k_j} =
\mathfrak {DA}\prod_i \pp_i^{n_i}\qq_1^{k_1-c}\prod_{j>1}\qq_j^{k_j}
\]%
The pole of $s_1$ is a factor of $\mathfrak Q$ and therefore $s_1$ is
integral over $F[t]$. Further, by construction $\Aa$ and $\Dd$ are
integral divisors relatively prime to $\Pp$ and $\Bb$. A similar
argument applies to $s_2$.%

Multiplying \eqref{eq:1.3} through by $b^{p^a}$ we will obtain the following equation.
\begin{equation}
\label{eq:1.4}%
t(s_1^{p^a} - b^{p^a-1}s_1) = s_2^{p^a} - b^{p^a-1}s_2.%
\end{equation}%
We can now rewrite (\ref{eq:1.4}) in the form%
\begin{equation}%
\label{eq:1.5}%
(s_1^{p^a}t - s_2^{p^a}) = b^{p^a-1}(s_1t - s_2).%
\end{equation}%
If $\ttt$ is any prime factor of $\Bb$ in $M$, then $\ttt$ does not
ramify in the extension $M/F(t)$, and since $p^a > 2$, we know that
$\ord_{\ttt}(s_1^{p^a}t - s_2^{p^a}) \geq 2$.  Further, we also have by Corollary \ref{cornoram}
\[%
\ord_{\ttt}\frac{d(s_1^{p^a}t - s_2^{p^a})}{dt} > 0.%
\]%
Finally,
\[%
 \ord_{\ttt}\frac{d(s_1^{p^a}t -s_2^{p^a})}{dt} = \ord_{\ttt}(s_1^{p^a}).%
\]%
Therefore, $s_1$ has a zero at $\ttt$. This, however, is impossible by
construction of $s_1$ as described above. Consequently, $\Bb$ is a
trivial divisor and $\mu$ and $\sigma$ are constants since their pole
divisor is trivial.  Now \eqref{eq:1.3} is implying that $t$ times a
constant is equal to a constant.  This can happen only if both
constants are zero.
\end{proof}

\begin{lemma}[\cite{Sh34}, Lemma 8.4.4]%
\label{le:reducepower}%
Let $s \in \Z_{>0}$. Let $x, v \in M\setminus\{0\}$ and assume that for some $\tilde{v}\in M$ we have that
$\tilde{v}^{p^a} = v$. Let $\displaystyle u = \frac{x^p + t}{x^p - t}$ if $p >2$ and let $\displaystyle u=\frac{x^2 +t^2+ t}{x^2 + t}$, if $p
=2$. Further, assume that%
 \begin{equation}%
 \label{eq:pthu}%
 \begin{array}{c}%
 \exists \mu_{i,j_i,l,j_l,z,m}, \sigma_{i,j_i,l,j_l}, \lambda_1, \lambda_{-1} \in M\\
 \forall i \exists j_i \forall(l\not=i)\exists j_l \forall m\forall z:\\%
 D(s,i,j_i, l,j_l,m,z,u, v,\mu_{i,j_i,l,j_l,z,m}, \sigma_{i,j_i,l,j_l}, \lambda_1, \lambda_{-1})%
\end{array}%
\end{equation}%
holds.
Then%
\begin{equation}%
\label{eq:minus1}
\begin{array}{c}%
\exists \tilde{\mu}_{i,j_i,l,j_l,z,m}, \tilde{\sigma}_{i,j_i,l,j_l},\tilde \lambda_1, \tilde \lambda_{-1} \in M\\%
\forall i \exists j_i \forall(l\not=i)\exists j_l \forall m\forall z:\\%
D(s-1,i,j_i, l,j_l,m,z,u, \tilde{v},\tilde{\mu}_{i,j_i,l,j_l,z,m}, \tilde{\sigma}_{i,j_i,l,j_l},\tilde{\nu}_{i,j_i,z}, \tilde \lambda_1, \tilde \lambda_{-1})%
\end{array}%
\end{equation}%
holds.
\end{lemma}%
\begin{lemma}[\cite{Sh34}, Lemma 8.4.5, Corollary 8.4.6 and
  \cite{Eis}, Lemma 3.9]%
\label{lemma:morepth}%
Let $s \in \Z_{\geq 0}, x, v \in M\setminus\{0\}$. Let $\displaystyle u = \frac{x^p  + t}{x^p - t}$, if $p >2$, and let $\displaystyle u=\frac{x^2 +t^2+ t}{x^2 + t}$, if
$p =2$. Further, assume that (\ref{eq:pthu}) holds. Then $v = u^{p^{as}}$.  
\end{lemma}
\begin{proof}
  First of all, we claim that for all $i, l$, it is the case that
  $u_{i,l}$ has no multiple zeros or poles except possibly at the
  primes with factors ramifying in $M/F(t)$, or poles or zeros of
  $t$. Indeed, all the poles of $u_{i,l}$ are zeros of $u + c_l$ and
  all the zeros of $u_{i,l}$ are zeros $u+c_i$. However, by Lemma
  \ref{le:simplezero} and by assumption on $c_i$ and $c_l$, all the
  zeros of $u + c_l$ and $u + c_i$ are simple, except possibly for
  zeros at the primes which are zeros or poles of $t$ or have factors
  ramifying in the extension $M/F(t)$.

We will show that if $s > 0$ then $v$ is a $p^a$-th power in $M$, and if $s=0$ then $u=v$. This assertion together
with Lemma \ref{le:reducepower} will produce the desired conclusion.

Note that by Corollary \ref{cor:C(w)}, we can choose distinct natural
numbers 
\[
i, l_1,\ldots,l_{k+1} \in \{0,\ldots,C_5\}
\]
 such that 
 \[
 \{c_i,
c_{l_1},\ldots,c_{l_{k+1}}\} \subset C_v \cap C_u 
\]
 and for all $1
\leq j_i \leq r_i,1 \leq j_{l_f} \leq r_{l_f}$, with $f=1,\ldots, k+1$,
we have that $u_{i,l_f}$ and $v_{i,j_i,l_f, j_{l_f}}$ have no
zeros or poles at the primes of $M$ with factors ramifying in the
extension $M/F(t)$, or primes occurring in the $M$-divisor of
$t$. Note also that for thus selected indices, all the poles and zeros
of $u_{i,l_f}$ are simple. We now proceed to pick natural numbers $i,
l_1,\ldots,l_{k+1},j_i,j_{l_1},\ldots,j_{l_{k+1}}$ such that the
equations in (\ref{eq:duf2.1}) - (\ref{eq:2.4}) are satisfied for
these values of indices, and $u_{i,l_1}$, $v_{i,j_i,l_1, j_{l_1}},
\ldots$, $u_{i,l_{k+1}}$, $v_{i,j_i,l_{k+1}, j_{l_{k+1}}}$ have no
poles or zeros at primes with factors ramifying in the extension
$M/F(t)$, or at primes occurring in the $M$-divisor of $t$.\\

Now assume $s > 0$, and let  $f$ range over the set $\{1,\ldots,k+1\}$.  First  let $z=\pm1$, while $m=0$, and
consider the two versions of the equation in (\ref{eq:2.3}) with these values of $z$ and $m$.
\begin{equation}
\label{eq:mod0.1}%
 v_{i,j_i,l_f,j_{l_f}}^2 - u_{i,l_f}^2 = \mu_{i,j_i,l_f,j_{l_f},1,0}^{p^a} -\mu_{i,j_i,l_f,j_{l_f},1,0},%
\end{equation}
\begin{equation}
\label{eq:mod0}%
 v_{i,j_i,l_f,j_{l_f}}^{-2} - u_{i,l_f}^{-2} = \mu_{i,j_i,l_f,j_{l_f},-1,0}^{p^a} -\mu_{i,j_i,l_f,j_{l_f},-1,0},%
\end{equation}%
Here either for all $f = 1,\ldots,k+1$, the divisor of
$v_{i,j_i,l_f,j_{l_f}}$ in $M$ is a $p^a$-th power of another divisor,
or for some $f$ and some prime {\te t} without factors ramifying in
the extension $M/F(t)$ and not occurring in the $M$-divisor of $t$
we have that $\ord_{\mte{t}}v_{i,j_i,l_f,j_{l_f}} = \pm 1$.

In the first case, given the assumption that $v_{i,j_i,l_f,j_{l_f}}$'s
do not have poles or zeros at ramifying primes and Lemma
\ref{lemma:divpth}, we have that $v$ is a $p^a$-th power in $M$.

So suppose the second alternative holds. In this case, without loss of
generality, assume {\te t} is a pole of $v_{i,j_i,l_f,j_{l_f}}$ for
some $f$. Next consider the following equations
\begin{equation}
\label{eq:mod1}%
 v_{i,j_i,l_f,j_{l_f}}^2t^{p^{as}} - u_{i,l_f}^2t = \mu_{i,j_i,l_f,j_{l_f},1,1}^{p^a} -\mu_{i,j_i,l_f,j_{l_f},1,1},%
\end{equation}
\begin{equation}
\label{eq:mod2}%
 v_{i,j_i,l_f,j_{k_f}}^2 - u_{i,l_f}^2 = \mu_{i,j_i,l_f,j_{l_f},0,1}^{p^a} -\mu_{i,j_i,l_f,j_{l_f},0,1},%
\end{equation}%
obtained from (\ref{eq:2.3}) by first making $z = 1, m = 1$ and then $z= 1, m = 0$. (If {\te t} were a zero of
$v_{i,j_i,l_f,j_{l_f}}$, then we would set $z$ equal to -1 in both equations.) Since $t$ does not have a  pole or
 zero at {\te t} and $p^a >2$, we must conclude that%
\[%
\ord_{\mte{t}}(v_{i,j_i,l_f,j_{l_f}}^2t^{p^{as}} - u_{i,l_f}^2t) =
\ord_{\mte{t}}(\mu_{i,j_i,l_f,j_{l_f},1,1}^{p^a}- \mu_{i,j_i,l_f,j_{l_f},1,1}) \geq 0%
\]%
and%
\[%
\ord_{\mte{ t}}(v_{i,j_i,l_f,j_{l_f}}^2 - u_{i,l_f}^2) = \ord_{\mte{t}}(\mu_{i,j_i,l_f,j_{l_f},0,1}^{p^a} -
\mu_{i,j_i,l_f,j_{l_f},0,1}) \geq 0%
\]%
Thus,%
\[%
\ord_{\mte{t}}v_{i,j_i,l_f,j_{l_f}}^2(t^{p^{as}} - t)
\]%
\[%
= \mbox{ord}_{\mbox{\te t}}(\mu_{i,j_i,l_f,j_{l_f},1,1}^{p^a} -\mu_{i,j_i,l_f,j_{l_f},1,1} -
t\mu_{i,j_i,l_f,j_{l_f},0,1}^{p^a} + t\mu_{i,j_i,l_f,j_{l_f},0,1}) \geq 0.%
\]%
 Finally, we must deduce that
$\ord_{\mte{t}}(t^{p^{as}} - t) \geq 2|\ord_{\mte{t}}v|$. But in $F(t)$ all the zeros of $(t^{p^{as}} - t)$ are
simple. Thus, this function can have multiple zeros only at primes ramifying in the extension $M/F(t)$. By
assumption {\te t} is not one of these primes and thus we have a contradiction, unless $v$ is a $p^a$-th power.

Suppose now that $s = 0$. Set $e=1$ again and let $i,l_1,\ldots,l_{k+1}$ be selected as above. Then from
(\ref{eq:mod1}) and (\ref{eq:mod2}) we obtain for $l_f \in \{l_1,\ldots,l_{k+1}\}$,
\[%
\mu_{i,j_i,l_f,j_{l_f},1,1}^{p^a} - \mu_{i,j_i,l_f,j_{l_f},1,1} = t(\mu_{i,j_i,l_f,j_{l_f},0,1}^{p^a} - \mu_{i,j_i,l_f,j_{l_f},0,1}).%
\]%
Note here that all the poles of $\mu_{i,j_i,l_f,j_{l_f},1,1}$ and
$\mu_{i,j_i,l_f,j_{l_f},0,1}$ are poles of $u_{i,l_f}$,
$v_{i,j_i,l_f,j_{l_f}}$ or $t$, and thus are not at any valuation
ramifying in the extension $M/F(t)$. From Lemma \ref{lemma:eqzero} and
equation (\ref{eq:mod2}) we
can then conclude that for all $l_f \in \{l_1,\ldots,l_{k+1}\}$%
\[%
 v_{i,j_i,l_f,j_{l_f}}^2 -u_{i,l_f}^2 = 0.
 \]%
Thus, $v_{i,j_i,l_f,j_{l_f}} = \pm u_{i,l_f}$. Since all the poles of $u_{i,l_f}$ are simple, (\ref{eq:2.4})  rules out "-". Therefore,
\begin{equation}%
\label{eq:2.6}%
v_{i,j_i,l_f,j_{l_f}} =  u_{i,l_f}.
\end{equation}%
Rewriting (\ref{eq:2.6}) we obtain
\[
 \frac{d_{i,j_i} - d_{l_f,j_{l_f}}}{v+d_{l_f,j_{l_f}}} = \frac{c_i - c_{l_f}}{u+c_{l_f}},
\]
or
\begin{equation}%
\label{eq:sup1} v = au+b,%
\end{equation}%
where $a,b$ are constants. However, unless $b=0$, we have a
contradiction with (\ref{eq:2.5.1})  because, unless $b=0$, we
have that $v^{-1}$ and $u^{-1}$ have different, and in the case of $u$, always simple poles.
Finally, if $a \not =1$, then we have a contradiction with (\ref{eq:2.5})  because the difference, unless it is 0 (and therefore
$a=1$), will have simple poles.

\end{proof}
\subsection{Satisfying equations}
\label{subsec:cansatisfy}
We now address address the issue we have avoided so far: satisfying the equations constituting our Diophantine definitions.  Before we proceed we introduce one more notation.
\begin{notation}
Let $F_1 =\F_p(C(F))$.
\end{notation}
\begin{lemma}
\label{satisfy}
If $w = t^{p^{as}}, s \in \Z_{\geq 0}$ then Equations \eqref{eq:alreadydown0} can be satisfied over $\F_p(t)$ and Equations \eqref{eq:alreadydown1} can be satisfied over $F_1(t)$.  Further, if  $v = u^{p^{as}}$ then Equations \eqref{eq:pthu} can be satisfied over $F_1(t)$.
\end{lemma}
\begin{proof}
We start with an elementary equality which constitute the basis of all the constructions in this section.
\begin{equation}%
\label{eq:210}%
x^{p^{as}} - x = (x^{p^{a(s-1)}} + x^{p^{a(s-2)}} + \dots + x)^{p^a} - (x^{p^{a(s-1)}} + x^{p^{a(s-2)}} + \dots + x)%
\end{equation}%
To satisfy  Equations \eqref{eq:alreadydown0}, it is enough to note that \eqref{eq:210} holds over $\F_p(x)$.  To satisfy Equations \eqref{eq:alreadydown1}, it is enough to make sure that if $ w = t^{p^{as}}, s \in \Z_{\geq 0}$, then for all $i, j$ there exist $b \in V_i, b' \in V_j$ such that $\displaystyle \frac{w+b}{w+b'}=\left (\frac{t+c_i}{t+c_j}\right )^{p^{as}}$.  This fact, however, follows immediately from the definition of $V_i$ and $V_j$ which contain all the $p$-th powers of $c_i$ and $c_j$ respectively.

Assuming $v = u^{p^{as}}$,  for
some $1 \leq j_i \leq r_i, 1 \leq j_k \leq r_k$, we have $v_{i,j_i,k,j_k} = (u_{i,k})^{p^{as}}$ for the same reason, since $|V_i|=r_i$ and $|V_j|=r_j$.
\end{proof}

\subsection{Defining $p$-th powers of  arbitrary elements}
\label{all}
We are now ready for the last sequence of propositions concluding the
proof.  We will have to separate the case of $p=2$ again.  We start
with the case of $p>2$.

\begin{proposition}[\cite{Sh34}, Proposition 8.4.8]%
\label{prop:pithx}%
Let $p >2$. Let $x, y \in M$.  Then there exist $v, \tilde{v}, u,
\tilde{u}, v_1, \tilde{v}_1$, $ u_1, \tilde{u}_1 \in M$, $s,i,j, r_1, j_1
\in \Z_{\geq0}$ such that
\begin{equation}%
\label{eq:E}
\left \{ \begin{array}{c}%
E(u,\tilde{u},v, \tilde{v}, x, y, j,i,s)\\
E(u_1,\tilde{u}_1,v_1, \tilde{v}_1, x+1, y+1, j_1,r_1,s)%
\end{array}
\right .%
\end{equation}%
hold if and only if $y=x^{p^s}$.
\end{proposition}
\comment{
\begin{proof}%
  Suppose (\ref{eq:E}) is satisfied over $K$.  Then using the fact
  that $E(u,\tilde{u},v, \tilde{v}, x, y, j,i,s)$ holds, from
  (\ref{eq:final1}), (\ref{eq:final3}) and (\ref{eq:final5}) we obtain
\[%
\frac{x^{p^{i+1}}-t^{p^i}}{x^{p^{i+1}}+t^{p^i}}=\frac{y^p-t^{p^s}}{y^p+t^{p^s}}
\]%
and
\[%
y=x^{p^i}t^{p^{s-1}-p^{i-1}}.
\]%
Similarly, from (\ref{eq:final2}), (\ref{eq:final4}) and (\ref{eq:final6})
\[%
y=x^{p^j}t^{-p^{s-1}+p^{j-1}}.%
\]%
Thus, $x^{p^j-p^i}=t^{2p^{s-1}-p^{j-1}-p^{i-1}}$.  From
$E(u_1,\tilde{u}_1,v_1, \tilde{v}_1, x+1, y+1, j_1,r_1,s)$ we
similarly conclude
\[%
(y+1)=(x+1)^{p^{r_1}}t^{p^{s-1}-p^{r_1-1}},
\]%
and
\[%
(x+1)^{p^{j_1}-p^{r_1}}=t^{2p^{s-1}-p^{j_1-1}-p^{r_1-1}}.
\]%

If $x$ is a constant, then $s=i=s=j_1=r_1$ and $y=x^{p^s}$. Suppose
$x$ is not a constant. If $2p^{s-1}-p^{j-1}-p^{i-1} >0$, then $x$ has
a zero at $\mte{p}$ and a pole at $\mte{q}$. Further, we also conclude
that $x+1$ has a pole at $\mte{q}$, $2p^{s-1}-p^{j_1-1}-p^{r_1-1} >0$
and $x+1$ has a zero at $\mte{p}$, which is impossible. We can
similarly rule out the case of $2p^{s-1}-p^{j-1}-p^{i-1} <0$. Thus,
$2p^{s-1}-p^{j-1}-p^{i-1} =0$, $s=i=s=j_1=r_1$, and $y=x^{p^s}$. On
the other hand if $y = x^{p^s}$ and we set $s=i=s=j_1=r_1$ we can
certainly find $v, \tilde{v}, u, \tilde{u}, v_1, \tilde{v}_1, u_1,
\tilde{u}_1 \in K$ to satisfy (\ref{eq:E}).
\end{proof}%
}
The following propositions treat the characteristic 2 case.%
\begin{lemma}[\cite{Sh34}, Proposition 8.4.9]
\label{le:p2}%
Let $p=2$.  Then for $x, y=\tilde{y}^2 \in M, j,r,s \in \Z_{\geq
  0}\setminus\{0\}, u,\tilde{u} \in M$ there exist $v, \tilde{v} \in
M$ such that %
\begin{equation}%
\label{eq:E2}%
E2(u,\tilde{u},v, \tilde{v}, x, y,j,r,s)\\%
\end{equation}%
holds if and only if there exist $v_1, \tilde{v}_1 \in M$ such that
\begin{equation}%
\label{eq:E2.1}%
E2(u_1,\tilde{u}_1,v_1, \tilde{v}_1, x, \tilde{y},j-1,r-1,s-1)\\%
\end{equation}%
holds.
\end{lemma}
\comment{
\begin{proof}%
Suppose that for some  $x, y=\tilde{y}^2 \in M, j,r,s \in \Z_{>0}, u,\tilde{u} \in M$, (\ref{eq:E2})
holds. Then from (\ref{eq:final52}) we derive
\begin{equation}%
\label{eq:vsquared}%
v=\frac{\tilde{y}^4+t^{2^{s+1}}+ t^{2^s}}{\tilde{y}^4+t^{2^s}}=\left(\frac{\tilde{y}^2+t^{2^s}+
t^{2^{s-1}}}{\tilde{y}^2+t^{2^{s-1}}}\right )^2.%
\end{equation}%
Thus, if we set %
\begin{equation}
\label{eq:E2.2}
v_1 = \frac{\tilde{y}^2+t^{2^s}+t^{2^{s-1}}}{\tilde{y}^2+t^{2^{s-1}}},%
\end{equation}%
we conclude that $v_1=u^{2^{j-1}}$.  Similarly, if we set
\begin{equation}%
\label{eq:E2.3}
\tilde{v}_1 = \frac{\tilde{y}^2+t^{-2^s}+t^{-2^{s-1}}}{\tilde{y}^2+t^{-2^{s-1}}},%
\end{equation}%
the $\tilde{v}_1=\tilde{u}^{2^{r-1}}$.  Thus, (\ref{eq:E2.1}) holds.  On the other hand, it is clear that if for
some $x, y=\tilde{y}^2 \in M, j,r,s \in \Z_{>0}, u,\tilde{u} \in M$ there exist $v_1, \tilde{v}_1 \in K$
such that (\ref{eq:E2.1}) holds, then by setting $v=v_1^2, \tilde{v}=\tilde{v}_1^2$ we will insure that
(\ref{eq:E2}) holds.
\end{proof}
}
\begin{proposition}[\cite{Sh34}, Proposition 8.4.10 and \cite{Eis},
  Theorem 3.1]
\label{prop:pithx2}%
Let $p=2$. Then for $x, y \in M, s \in \Z_{\geq 0}$ there exist $j,r \in \Z_{\geq 0}$, $u,\tilde{u},v, \tilde{v} \in M$ such that
(\ref{eq:E2}) holds if and only if $y=x^{2^s}$.%
\end{proposition}%
\comment{
\begin{proof}%
 Since $E2(u,\tilde{u},v, \tilde{v}, x,y,j,r,s)$ holds, from (\ref{eq:final12}), (\ref{eq:final32}) and (\ref{eq:final52}) we conclude that
\begin{equation}%
\label{eq:final82}%
y^2=(x^{2^{r+1}}t^{2^{s+1}} + t^{2^r+2^{s+1}}+t^{2^r+2^{s+1}})t^{-2^{r+1}}.%
\end{equation}%
Similarly, from (\ref{eq:final12}), (\ref{eq:final32}) and (\ref{eq:final52}) we conclude that
\begin{equation}%
\label{eq:final92}%
y^2=(x^{2^{j+1}}t^{-2^{s+1}} + t^{-2^j-2^{s+1}}+t^{-2^j-2^{s+1}})t^{2^{j+1}}.%
\end{equation}
Thus to show that (\ref{eq:E2}) implies $y=x^{2^s}$ it is enough to show that $s=r$ or $s=j$.

By Lemma \ref{le:p2} it is enough to consider two cases: one of $r, j,
s$ is equal to zero or $y$ is not a square. First suppose that
$s=0$. Then $v=\frac{y^2+t^2+t}{y^2+t}$ and $v$ is not a square since
$\frac{dv}{dt}=\frac{t^2}{y^4+t^2}\not=0$.  But $v = u^{2^r}$ and
therefore $r=0=s$.

Suppose now that $r=0$.  Then $v=u$ and therefore $v$ is not a square
by a an argument similar to the one above.  On the other hand, if $s
>0$ the $v$ is a square.  Thus, to avoid contradiction we must
conclude that $s=0$.

Finally if $j=0$, then $\tilde{v}=\tilde{u}$ is not a square and $s=0$
again.  Thus we have reduced the problem to the case where $y$ is not
a square and $r, s, j$ are positive.  In this case, from
(\ref{eq:final82}), taking the square root of both sides of the
equation we obtain
\begin{equation}%
\label{eq:final82-1}%
y=(x^{2^r}t^{2^s} + t^{2^{r-1}+2^s}+t^{2^r+2^{s-1}})t^{-2^r}%
\end{equation}
implying that, unless $r=1, s>1 $ or $s=1, r>1$, $y$ is a square.

Similarly, from (\ref{eq:final92}), taking the square root we obtain,
\begin{equation}%
\label{eq:final92-1}%
y=(x^{2^j}t^{-2^s} + t^{-2^{j-1}-2^s}+t^{-2^j-2^{s-1}})t^{2^j},%
\end{equation}%
 implying that, unless $j=1, s>1$ or $s=1, j>1$, $y$ is a square again.

Thus, either $s=1, r > 1, j>1$ or $r=j=1, s>1$, or $y$ is a square.
First suppose $s=1$.  Then  eliminating $y$ from (\ref{eq:final82-1}) and (\ref{eq:final92-1}) and
substituting 1 for $s$ yields
\[%
(x^{2^r}t^2 + t^{2^{r-1}+2}+t^{2^r+1})t^{-2^r}=(x^{2^j}t^{-2} + t^{-2^{j-1}-2}+t^{-2^j-1})t^{2^j},
\]%
 \[%
x^{2^r}t^{2-2^r} + t^{-2^{r-1}+2}+t=x^{2^j}t^{2^j-2} + t^{2^{j-1}-2}+t^{-1},
\]%
Then $t+\frac{1}{t}$ is a square in $K$.  This is impossible, since this element is not a square in
$C_K(t)$ and the extension $K/C_K(t)$ is separable.

Suppose now that $r=j=1, s>1$.    Eliminating $y$ from (\ref{eq:final82-1}) and (\ref{eq:final92-1}) and
substituting 1 for $j$ and $r$ produces
\begin{equation}%
\label{eq:final102}%
t^{2^s-2}x^2 +t^{2-2^s}x^2 = t^{2^s-1} + t^{2^{s-1}} +  t^{1-2^s}+ t^{-2^{s-1}}.%
\end{equation} %
This equation implies that $t^{2^s-1} + t^{1-2^s}$ is a square in $K$.   But%
\[%
\ord_{\mte{p}}(t^{2^s-1} + t^{1-2^s})=\ord_{\mte{p}}t^{1-2^s}=1-2^s%
\]%
is odd and $\pp$ is not ramified in the extension $K/C_K(t)$. Thus we have a contradiction again.

Hence, if $y$ is not a square, $r=s=j=0$.  Thus, we have shown that (\ref{eq:E2}) implies $y=x^{2^s}$.  Conversely,
if $y=x^{2^s}$, we can set $j=r=s$ and (\ref{eq:E2}) will be satisfied.
\end{proof}
}

\subsection{Adjusting for arbitrary constant fields}
\label{arbit}
We will now describe how to adjust the arguments above to take care of
the case where the field of constants is not necessarily algebraically
closed.  So let $K$ be an arbitrary function field of positive
characteristic.  Let $M$ be the field obtained from $K$ by adjoining
the algebraic closure of the constant field of $K$ and as above denote
the constant field of $M$ by $F$.  Let $t \in M$ be a non-constant
element such that all of its poles and zeros are simple.  (As we have
seen above, such an element always exists.)  The element $t$ and all
other elements of $M$ are algebraic over $K$.  Given such an
element $t$, compute $C_5$ and construct $C(F)$.  Let
$\hat{K}=K(t,C(F))$. Then $\hat K/K$ is a finite extension.
Now all the equations discussed above have their coefficients in $\hat
K$ and can be satisfied over $\hat K$.  As far as solutions to these
equation go, we can always consider solutions in $M$ to make sure we
have only the solutions we want.  E.g., if we know that 
\eqref{eq:alreadydown0} and \eqref{eq:alreadydown1} imply that $w
=t^{p^{as}}$ when we are looking at all possible solutions in $M$,
then this is certainly true of $\hat K$.  Thus $P(\hat K)$ is
existentially definable over $\hat K$.  Finally, we appeal to
Corollary \ref{fixed}, to conclude that $P(K)$ is existentially
definable over $K$.

\section{Subsets of a field integral at a prime}
\label{sec:integralsubsets}
\setcounter{equation}{0} In this section we construct a Diophantine
definition of the set $\mbox{INT}( K, \ttt, t)$ for some prime $\ttt$
and some non-constant element $t$ of our function field.
Unfortunately, we have to modify somewhat the assumptions and notation for
this section.  The new notation and assumptions can be found
below. Also, as in the section on $p$-th powers, our initial
assumptions will include some conditions on the field, which might not
be true of the given field.  We will show however, that they can be
made true in a finite extension of the given field.
\begin{notationassumption}
\label{intnot}
$\left. \right.$
\begin{itemize}
\item $p, q$ will denote two not necessarily distinct rational primes.
\item $\F_p$ will denote a finite field of $p$ elements.
\item $K$ will denote a function field over a field of constants $C$
  of characteristic $p >0$.
\item  $t \in K$ will denote an element of the field such that $K/C(t)$ is finite and separable.
\item $C_0$ will denote the algebraic closure of $\F_p$ in $C$, and $K_0$ will denote the algebraic closure of $C_0(t)$ in $K$.  If $q \not= p$, then assume that $C_0$ contains the primitive $q$-th root of unity.
\item Let $\gamma \in K$ generate $K$ over $C(t)$ and let $\gamma_0$ generate $K_0$ over $C_0(t)$.
\item For  some $a \in C$ algebraic over $\F_p$, $K$ and $C$ will not contain any root of a polynomial 
\begin{equation}
\label{notp}
T^q-a,
\end{equation}
 in the case $q \not=p$, and any root of 
 \begin{equation}
 \label{q=p}
 T^p-T+a,
 \end{equation}
  in the case $q=p$.  Let $\alpha$ be a root of \eqref{notp}, if $q\not=p$ and let $\alpha$ be the root of \eqref{q=p} otherwise. 
\item Assume all the poles and zeros of $t$ in $K$ and $K_0$ are
  simple. In particular, $\Pp$, the zero of $t$ in $C(t)$ or $C_0(t)$
  is not ramified in $K/C(t)$ or $K_0/C_0(t)$.  Note also that if
  $\pp$ (or $\pp_0$) is a prime of $K$ ($K_0$ respectively) lying
  above $\Pp$, then $\ord_{\pp}t=1$ ($\ord_{\pp_0}t=1$ respectively).
\item Denote by $\Qq$ the pole divisor of $t$ in $K$ or $K_0$.
\item If $q \not = p$, let $b \in C_0\setminus\{0\}$ be such that for some $c \in C_0$ we have that $c^q=b$.
\item  For $w  \in K$, let $h_w=t^{-1}w^q +t^{-q}$.
\item If $p=q$, let $\beta_w$ be the root in $\tilde K$, the algebraic
  closure of $K$, of $\displaystyle T^p-T-\frac{1}{h_w}$. If
  $p\not=q$, set $\beta_w$ to be the root of $\displaystyle
  T^q-(\frac{1}{h_w}+1)$.
\item Let $\delta \in \tilde K$ be a root of the polynomial $T^p-T+t$,
  if $q=p$ and let $\alpha$ be a root of the polynomial $T^q-(t+1)$, if
  $q \not =p$.
\item Let $N=K(\delta)$ and let $N_0=K_0(\delta)$.
\end{itemize}
The diagram below shows  the field extensions we will consider in this section:
\begin{center}
\xymatrix{{N_0(\beta_w,\alpha)}\ar[r]&{N(\beta_w,\alpha)}\\
{ N_0(\beta_w)}\ar[r]\ar[u] &{N(\beta_w)}\ar[u]\\
{N_0=K_0(\delta)}\ar[u]\ar[r]&{N=K(\delta)}\ar[u]\\
{K_0=C_0(t,\gamma_0)}\ar[u]\ar[r]&{K=C(t,\gamma)}\ar[u]\\
{C_0(t)}\ar[u]\ar[r]&{C(t)}\ar[u]\\
{\F_p(t)}\ar[u]\ar[ru]&}

\end{center}
\end{notationassumption}
We start with some basic lemmas concerning function fields and local
fields.  The proofs of the facts in the first lemma can be found in
\cite[Ch.\ V, \S5, and Theorem 6.4]{LangAlgebra}.
\begin{lemma}
\label{basic}
The following statements are true.
\begin{itemize}
\item If $L$ is algebraic over a finite field of characteristic $p >0$
  and is not algebraically closed, then it has an extension of prime
  degree $q$.  Further, if $q \not =p$, then for some $a \in L$, the
  polynomial $X^q-a$ is irreducible and if $q=p$, then for some $a \in
  L$, the polynomial $X^p-X-a$ is irreducible.
\item All the solutions to $X^p-X-a=0$ in the algebraic closure of $L$
  can be written in the form $\alpha + i, i=0,\ldots,p-1$, where
  $\alpha$ is any root of the equation.
\item If $L$ is algebraic over a finite field of characteristic $p >0$ and is not algebraically closed, then any finite extension of $L$ is also not algebraically closed.
\end{itemize}
\end{lemma}
\begin{lemma}
  Let $G$ be a field of positive characteristic $p$ and let $q$ be a
  prime number.  If $q \not = p$, assume $G$ contains a primitive
  $q$-th root of unity.  Let $\alpha$ be an element of the algebraic
  closure of $G$, either in $G$ or of degree $q$ over $G$.  Let
  $\alpha_j=\alpha +j, j=0,\ldots, p-1$, if $p=q$, and let
  $\alpha_j=\xi_q^j\alpha, j=0,\ldots, q-1$, if $q\not=p$.  Let
\begin{equation}
\label{polynorm0}
P(a_0,\ldots,a_{q-1})=\prod_{j=0}^{q-1}(a_0 +a_1\alpha_j + \ldots +a_{q-1}\alpha_j^{q-1}).
\end{equation}
In this case, if $[G(\alpha):G]=q$, then
$P(a_0,\ldots,a_{q-1})={\mathbf N}_{G(\alpha)/G}(a_0 +a_1\alpha +
\ldots +a_{q-1}\alpha^{q-1})$.  At the same time, if $\alpha \in G$,
then for any $y \in G$ the equation $P(X_0,\ldots,X_{q-1})=y$ has
solutions $x_0,\ldots,x_{q-1}\in G$.
\end{lemma}
\begin{proof}
  The only assertion of the lemma which requires an argument is the
  assertion that for any $y \in G$ the equation
  $P(X_0,\ldots,X_{q-1})=y$ has solutions $x_0,\ldots,x_{q-1}\in G$
  assuming $\alpha \in G$.  To see that consider the following linear
  system of equations in $a_0,\ldots, a_{q-1}$:
\begin{equation}
\label{sysinL}
\sum_{i=0}^{q-1}a_i\alpha_j^i=y_j, j=1,\ldots,q,
\end{equation}
where $y_1=y$ and $y_j=1, j=2,\ldots,q$.  Observe that the determinant
of the system is a Vandermonde determinant and thus is non-zero.
Hence the system has solutions.  By Cramer's rule, all
the solutions will be in $G$.
\end{proof}
\begin{lemma}
\label{lemma:notcongruent}
Let $G/H$ be a Galois extension of algebraic function fields of degree
$n$.  Let {\te p} be a prime of $H$ which does not split in $G$.  Let
$x \in H$ be such that ord$_{\ttt}x \not \equiv 0 \pmod n$.  Then $x$ is not a norm of an element of $G$.
\end{lemma}
\begin{proof}
  Let $y=y_1,\ldots, y_n \in G$ be all the conjugates of a $G$-element
  $y$ over $H$.  Let $\mathfrak T$ be the prime above $\ttt$ in $G$.
  In this case, $\ord_{\mathfrak T}y_i = \ord_{\mathfrak T}y_j$ for
  all $i,j = 1,\ldots,n$.  Therefore, $\ord_{\mathfrak T}{\mathbf
    N}_{G/H}(y)= \sum_{i=1}^n \ord_{\mathfrak T}y_i =n\ord_{\mathfrak
    T}y \equiv 0 \mod n$. At the same time we have that
  $\ord_{\mathfrak T}{\mathbf N}_{G/H}(y) =\ord_{\ttt}{\mathbf
    N}_{G/H}(y)$ and the conclusion of the lemma follows.
\end{proof}
\begin{lemma}
\label{lemma:unramified}
Let $H/F$ be an unramified extension of local fields of degree $n$.
Let {\te t} be the prime of $F$.  Let $x \in F$ be such that
$\mbox{ord}_{\mbox{\te t}}x \equiv 0 \pmod n$.  Then $x$ is a norm of some element of $H$.
\end{lemma}
\begin{proof}
Let $\pi$ be a local uniformizing parameter for {\te t}.  Then $x = \pi^n\varepsilon$, where $\varepsilon$ is a unit.  Since $\pi^n$ is an $F$-norm, $x$ is an $F$-norm if and only if $\varepsilon$ is an $F$ norm.  The last statement is true by \cite[Corollary, page 226]{W}.
\end{proof}

We now consider the ramification behavior of a given set of primes in an extension.
\begin{lemma}
\label{lemma:allsol}
Let $L$ be a function field of characteristic $p$, let $v \in L$ and let $\delta$ be a root of the equation
\begin{equation}
\label{eq:5}
x^p - x - v = 0.
\end{equation}
In this case either $\delta \in L$ or $\delta$ is of degree $p$ over
$L$.  In the second case the extension $L(\delta)/L$ is cyclic of
degree $p$ and the only primes possibly ramified in this extension are
the poles of $v$.  More precisely, if for some $L$-prime {\te a},
$\mbox{ord}_{\mbox{\te a}}v \not \equiv 0$ modulo $p$ and
$\ord_{\mte{a}}v < 0$, then a factor of {\te a} in $L(\delta)$ will be
ramified completely.  At the same time all zeros of $v$ will split
completely, i.e.\ into factors of relative degree 1, in $L(\delta)$.
\end{lemma}
\begin{proof}
  Let $\delta = \delta_1,\ldots,\delta_p$ be all the roots of
  (\ref{eq:5}) in the algebraic closure of $L$.  Then we can number
  the roots so that $\delta_i = \delta + i - 1$.  Thus, either the
  left side of (\ref{eq:5}) factors completely or it is irreducible.
  In the second case $\delta$ is of degree $p$ over $L$ and
  $L(\delta)$ contains all the conjugates of $\delta$ over $L$.  Thus,
  the extension $L(\delta)/L$ is Galois of degree $p$, and therefore
  cyclic.  Next consider the different of $\delta$.  This different
  is a constant.  By \cite[Lemma 2, page 71]{C}, this implies that no
  prime of $L$ at which $\delta$ is integral has any ramified factors
  in the extension $L(\delta)/L$.  Suppose now {\te a} is a prime of
  $L$ described in the statement of the lemma.  Let $\tilde{\mbox{\te
      a}}$ be an $L(\delta)$-prime above {\te a}.  Then
  $\mbox{ord}_{\tilde{\mbox{\te a}}}v \equiv 0 $ modulo $p$.  Thus,
  $\tilde{\mbox{\te a}}$ must be totally ramified over {\te
    a}. Finally, let $\bb$ be zero of $v$.  Since the power basis of
  $\delta$ has a constant discriminant, the power basis of $\delta$ is
  an integral basis with respect to $\bb$ and therefore, if the
  irreducible polynomial of $v$ factors completely modulo $\bb$, then
  $\bb$ factors completely in the extension.
\end{proof}
In a similar manner one can show that the following lemma is true.
\begin{lemma}
\label{ramifyorsplit}
Let $L$ be a function field of characteristic $p>0$ possessing a $q$-th
primitive root of unity, when $q \not=p$, $z \in L$, $\gamma$ a root
of $T^q-z$, $\aaa$ a prime of $L$.  In this case, if $\ord_{\aaa}z
\not \equiv 0 \mod q$, then $\aaa$ is completely ramified in the
extension $L(\gamma)/L$.  If $z$ is integral at $\aaa$ and $z \equiv
c^q \not =0 \mod \aaa$, then $\aaa$ splits completely, i.e.\ into
factors of relative degree 1, in the extension $L(\gamma)/L$.
\end{lemma}

We now specialize the lemmas above to our situation.  
\begin{corollary}
\label{K_1}
The following statements are true about the extensions $N/K$ and $N_0/K_0$.
\begin{enumerate}
\item There is no constant field extension.
\item The factors of $\Pp$ split completely, i.e into factors of relative degree 1.
\item The factors of $\Qq$ are ramified completely, i.e.\ into factors of relative degree 1.
\end{enumerate}

\end{corollary}

Next we need take a look at zeros and poles of $h_w$ and zeros and poles of $w$ in $N, N_0, N(\beta_w),$ and $N(\beta_w)$.  (We remind the reader that $h_w=t^{-1}w^q +t^{-q}$,  $\beta_w$ is a root of $T^p-T-\frac{1}{h_w}$ if $p=q$, and $\beta_w$ is a root of $T^q-(\frac{1}{h_w}+1)$ if $p\not=q$.)
\begin{lemma}
\label{order}
The following statements are true in $N(\beta_w)$ and, assuming $w$ is algebraic over $C_0(t)$, in $N_0(\beta_w)$:
\begin{enumerate}
\item If $\hat \pp$ is a prime of $N(\beta_w)$ (or $N_0(\beta_w)$) and
  $\hat \pp | \Pp$, while $\ord_{\hat \pp}w<0$, then the relative
  degree of $\hat \pp$ over $\bar \pp$, the prime below it in $N$
  (respectively $N_0$) is 1, and therefore the relative degree of
  $\hat \pp$ over $\pp$, the prime below it in $K$ (respectively
  $K_0$) is 1.
\item If $\hat \pp$ is a prime of $N(\beta_w)$ (or $N_0(\beta_w))$ and
  $\hat \pp | \Pp$ in $N(\beta_w)$ (respectively $N_0(\beta_w)$) while
  $\ord_{\hat \pp}w<0$ then $\ord_{\hat \pp}h_w <0$ and $\ord_{\hat
    \pp}h_w \not \equiv 0 \mod q$.

\item If $\ttt$ is a prime of $N(\beta_w)$ (or $N_0(\beta_w)$) and
  $\ttt \not | \Pp$, then $\ord_{\ttt}h_w\equiv 0 \mod q$.
\item If $\pp$ is a prime of $K$ (or $K_0$) such that $\pp | \Pp$ and
  $\ord_{\pp}w \geq 0$ then $\ord_{\pp}h_w \equiv 0 \mod q$.
\end{enumerate}
\end{lemma}
\begin{center}
\xymatrix{{N_0(\beta_w)}\ar[r] &{N(\beta_w)}\\
{N_0=K_0(\delta)}\ar[u]\ar[r]&{K(\delta)}\ar[u]\\
{K_0}\ar[u]\ar[r]&{K}\ar[u]}

\end{center}
\begin{proof}
  First let $\pp |\Pp$ in $K$ (or $K_0$) and note that by Corollary
  \ref{K_1}, we have that $\pp$ will split completely into factors of
  relative degree 1 in $N$ (respectively $N_0$).  Next, if
  $\ord_{\pp}w <0$, then $\ord_{\pp} h_w <0$ and $\ord_{\pp} h_w \not
  \equiv 0 \mod q$.  Further, for
  any $\bar \pp | \pp$ (in $N$ or $N_0$) we have $\ord_{\bar
    \pp}h_w=\ord_{\pp}h_w$, since $\bar \pp$ splits completely in the
  extension $N/K$ (or $N_0/K_0$).  Lemmas
  \ref{lemma:allsol} and \ref{ramifyorsplit} imply that $\bar \pp$
  will split completely in the extension $N(\beta_w)/N$ (respectively
  $N_0(\beta_w)/N_0$). So if $\hat \pp | \bar \pp$ we also have
  that $\ord_{\hat \pp}h_w=\ord_{\bar \pp}h_w=\ord_{\pp}h_w$.  At the
  same time, by the same argument, the relative degree of $\hat \pp$
  over $\pp$ is one.
 
  Again by Corollary~\ref{K_1},
we have that $\ord_{\hat \qq}h_w \equiv
  0 \mod q$ for any $N$-prime (or $N_0$-prime) $\hat \qq$ lying above
  a $K$-prime (or $K_0$-prime) $\qq$ dividing $\Qq$.  Finally consider
  the primes occurring in the divisor of $h_w$.  

If $\ttt$ is a pole
  of $h_w$ in $K$ (or $K_0$) and $\ttt$ does not occur in the divisor
  of $t$, then it is a pole of $w$ and $h_w$ has order divisible by
  $q$ at $\ttt$.  Hence, if $\hat \ttt$ is a prime of $N(\beta_w)$  (or
  $N_0(\beta_w)$) above $\bar \ttt$, we also have that $h_w$ has order divisible by $q$
  at $\hat \ttt$.  

Now let $\bar \ttt$ be a zero of $h_w$ in $N$ (or
  $N_0$) with order not divisible by $q$. Lemmas
  \ref{lemma:allsol} and \ref{ramifyorsplit} imply that
  $\bar \ttt$ ramifies completely in the extension $N(\beta_w)/N$ (or
  $N_0(\beta_w)/N_0$).  Finally, the last assertion of the lemma
  follows from the formula defining $h_w$.

\end{proof}

\begin{lemma}
\label{not split}
If $w$ has a pole at any factor of $\Pp$ in $K$, then there is no $x \in N(\alpha,\beta_w)$ such that 
\begin{equation}
\label{normeq}
{\mathbf N}_{N(\alpha, \beta_w)/N(\beta_w)}(x)=h_w
\end{equation}
\end{lemma}
\begin{proof}
  Let $\hat \pp$ be a factor of $\Pp$ in $N(\beta_w)$ such that $w$
  has a negative order at $\hat \pp$.  In this case, $w$ has a
  negative order at $\pp$, the prime below $\hat \pp$ in $K$.  By
  Lemma \ref{order}, $\pp$ splits completely into distinct unramified
  factors of relative degree 1 and $\pp$ is of degree 1 in $K$, so we
  conclude that there is no constant field extension in the extension
  $N(\beta_w)/N$ and  either \eqref{notp} or \eqref{q=p}, depending on whether $p=q$ or $p \not=q$, do not have a root in the residue field
  of $\hat \pp$ in $N(\beta_w)$.  Thus, $\hat \pp$ does not split in
  the extension $N(\beta_w,\alpha)/N(\beta_w)$. If $h_w$ is a norm in
  this extension it must have order divisible by $q$ at $\hat \pp$.
  However, by Lemma \ref{order} again, we have that $\ord_{\hat \pp}
  h_w =\ord_{\pp}h_w \not \equiv 0 \mod q$.
\end{proof}
\begin{lemma}
\label{hassolutios}
If $w$ is algebraic over $\F_p(t)$ and has no poles at any factor of $\Pp$, then there exists $x \in N(\alpha,\beta_w)$  satisfying \eqref{normeq}. 
\end{lemma}
\begin{proof}
First we observe that it is enough to find $x \in N_0(\alpha,\beta_w)$ with 
\begin{equation}
\label{algnormeq}
{\mathbf N}_{N_0(\alpha, \beta_w)/N_0(\beta_w)}(x)=h_w.
\end{equation}
Indeed, since $\alpha$ is of degree $q$ over both $N(\beta_w)$ and
$N_0(\beta_w)$ or is of degree 1 over both fields, any element $x \in
N_0(\alpha, \beta_w)$ has the same coordinates with respect to the
power basis of $\alpha$ (which is either $\{1\}$, if the degree of $\alpha$
over the fields in question is 1, or $\{1,\alpha,\ldots, q-1\}$, if the
degree of $\alpha$ over both fields is $q$).  Thus $x$ has the same
conjugates over $N(\beta_w)$ and $N_0(\beta_w)$ and therefore the same
norm.  Next, by Lemmas \ref{ramifyorsplit} and \ref{lemma:allsol}, and
by construction of $h_w$, the divisor of $h_w$ is a
$q$-th power of another divisor.  Now, if $\alpha \in N_0(\beta_w)$,
we are done.  Otherwise, we observe that there is a finite extension
$\hat N_0$ of $\F_p(t)$ such that
    \begin{itemize}
    \item $\alpha$ is of degree $q$ over $\hat N_0$,
    \item $w, h_w \in \hat N_0$,
    \item the divisor of $h_w$ is a $q$-th power of another divisor.
    \end{itemize}
   By an argument similar to the one above, it is enough to solve
   \begin{equation}
\label{smallalgnormeq}
{\mathbf N}_{\hat N_0(\alpha, \beta_w)/\hat N_0(\beta_w)}(x)=h_w.
\end{equation} 
Since the extension $\hat N_0(\alpha, \beta_w)/\hat N_0(\beta_w)$ is
unramified and thus locally every unit is a norm (see
\cite[Corollary, page 226]{W}), we conclude by the Strong Hasse Norm
principle (see \cite[Theorem 32.9]{Reiner}), that $h_w$ is a norm and
therefore $x$ as required exists.
\end{proof}
We now have the following theorem.
\begin{theorem}
\label{thm:polynorm}
Let $\alpha_j=\alpha +j, j=0,\ldots, p-1$, if $p=q$ and let $\alpha_j=\xi_q^j\alpha, j=0,\ldots, q-1$, if $q\not=p$.  Let 
\begin{equation}
\label{polynorm}
P(a_0,\ldots,a_{q-1})=\prod_{j=0}^{q-1}(a_0 +a_1\alpha_j + \ldots +a_{q-1}\alpha_j^{q-1})=h_w.
\end{equation}
In this case there exist $a_0,\ldots, a_{q-1} \in N(\beta_w)$ such
that \eqref{polynorm} holds only if $w$ has no poles at any factor
$\Pp$, and if $w$ is algebraic over $\F_p(t)$ and has no poles at any
factor of $\Pp$, then $a_0,\ldots, a_{q-1} \in N(\beta_w)$ such that
\eqref{polynorm} holds exist.
\end{theorem}
Now combining Theorem \ref{thm:polynorm} with Theorem \ref{bounddegree}
 we have the following result:
\begin{proposition}
The set $\mbox{INT}(K,\pp,t)$ is Diophantine over $K$.
\end{proposition}
We now add our result on definability of $p$-th powers to the proposition above to conclude that the following assertion is true:
\begin{proposition}
Hilbert's Tenth Problem is unsolvable over $K$.  (As above, if $K$ is uncountable we consider polynomials with coefficients in a finitely generated ring, more specifically the algebraic closure of $\F_p(t)$.)
\end{proposition}
\subsection{Removing the assumptions on $K$.}
Finally, we need to remove the assumptions we imposed on $K$.  We show
that given an arbitrary function field $G$ of positive characteristic
and not containing the algebraic closure of a finite field, we can a
find a finite extension $K$ of $G$ where all the assumptions above are
satisfied.  We proceed in several steps.
\begin{enumerate}
\item Let $M$ be the field obtained by adjoining to $G$ the algebraic closure
  $F$ of the constant field of $G$.  Since the constant field
  of $M$ is perfect, as in the section on $p$-th powers, we can find a
  non-constant element $z$ of $M$ such that $M/F(z)$ is separable,
  implying $z$ is not a $p$-th power in $M$.
\item Let $M_0$ be the algebraic closure in $M$ of $F_0(z)$, where
  $F_0$ is the algebraic closure of $\F_p$.  Observe that $z$ is not a
  $p$-th power in $M_0$ and hence the extension $M_0/F_0(z)$ is also
  separable.
\item Consider now the extensions $M/F(z)$ and $M_0/F_0(z)$.  Let $\gamma
  \in M, \gamma_0 \in M_0$ be such that $M=F(\gamma, z),
  M_0=F_0(\gamma_0,z)$.  Let $\Gamma \subset F$ be a finite set
  containing all the coefficients of the monic irreducible polynomials
  of $\gamma$ and $\gamma_0$ over $F(z)$ and $F_0(z)$, respectively.
\item Since both extensions $M/F(z)$ and $M_0/F_0(z)$ are finite
  and there are only finitely many ramified primes, we can find $c_1,
  c_2 \in F_0$ such that $\displaystyle t=\frac{w-c_1}{w-c_2}$ has
  only simple zeros in both $M$ and $M_0$.  We can also select
  $c_1, c_2$, so that $t$ does not have zeros or poles at the zeros of
  the discriminant of the power basis of $\gamma$ or $\gamma_0$ and
  $\gamma$ and $\gamma_0$ are both integral with respect to the zero
  and pole divisors of $t$.  Observe that $F_0(t)=F_0(z)$ and
  $F(t)=F(z)$.
\item Consider the monic irreducible polynomial of $\gamma$ (or
  $\gamma_0$) over $F(z)$ (or $F_0(z)$) modulo the zero divisor of $t$
  and also modulo the pole divisor of $t$.  Let $\Delta \subset F$ be
  a finite set containing all the roots of the reduced polynomials.
\item Set $K=G(t, \gamma, \gamma_0, \Gamma, \Delta)$ and add, if
  necessary, the primitive $q$-th roots of unity.
\end{enumerate}
As above, let $C$ and $C_0$ be the constant fields of $K_0$ and $K$,
respectively, and consider the extensions $K/C(t)$ and $K_0/C_0(t)$.
By construction, $\gamma$ has the same
monic irreducible polynomial over $C(t)$ and $F(t)$, and $\gamma_0$
has the same monic irreducible polynomial over
$C_0(t)$ and $F_0(t)$.  Since extensions $M/F(t)$ and $M_0/F_0(t)$ are
separable, all the roots of these polynomials are distinct. Hence
the extensions $K/C(t)$ and $K_0/C_0(t)$ are also separable.  Also by
construction, the pole divisor and the zero divisor of $t$ are prime to
the divisor of the discriminant of the power bases of $\gamma$ and
$\gamma_0$ and both $\gamma$ and $\gamma_0$ are integral with respect
to the zero and the pole divisor of $t$.  Consequently, the power
bases of $\gamma$ and $\gamma_0$ are both integral bases with respect
to the primes which are the pole and the zero of $t$ in $F(t)$ and
$F_0(t)$ respectively, and the pole and the zero of $t$ do not ramify
in the either extension.

Further, by \cite[Chapter 1, \S8, Proposition 25]{L}, the
factorization of the monic irreducible polynomials of $\gamma$ and
$\gamma_0$ corresponds to the factorization of the zero and the pole
of $t$ in $K$ and $K_0$, respectively.  However, by construction again,
these polynomials factor completely (into distinct factors) modulo the
zero and modulo the pole divisor of $t$.  So both primes will factor
into (unramified) factors of relative degree 1.  Since the pole
divisor and the zero divisor of $t$ are are also of degree 1, we must
conclude that their factors in $M$ and $M_0$ are also of degree 1.
Finally, we note that $C$ does not contain the algebraic closure of
$\F_p$ as was noted in Lemma \ref{basic}.

Now by Corollary \ref{probabove}, we can conclude that Hilbert's Tenth
Problem is unsolvable over $G$ (with the usual clarification in the
case $G$ is uncountable).  This concludes the proof of Theorems
\ref{countable} and \ref{uncountable}.

\section{First-order undecidability of function fields of positive
  characteristic}
\label{sec:firstorder}
Let $K$ be any function field of positive
characteristic and $t \in K$ an element with simple zeros and poles.
In Theorem 2.9 of \cite{ES} we showed that if $P(K,t)=\{x \in K: \exists s \in \Z_{\geq 0}, x=
t^{p^s}\}$ is first-order definable
over $K$, then $(\Z, |, +)$ has a model over $K$ in a first-order ring
language with finitely many parameters, and thus the first-order
theory of $K$ in a ring language with finitely many parameters is
undecidable.  The transition to the ring language without parameters
is in Section 5 of \cite{ES} and is not dependent on the nature of the
field.  Since we have now defined $P(K,t)$ existentially over any
function field of positive characteristic, the conclusion of Theorem
2.9 applies to any such field and so does the strengthening of the
result to the the first-order language without parameters.

\end{document}